\documentclass[preprint,12pt]{elsarticle}

\usepackage{amsmath}
\usepackage{amssymb}
\usepackage{amsfonts}
\usepackage{amsthm}
\usepackage{comment}
\usepackage{pgfplots}
\usepackage{graphicx}
\usepackage{amsmath}
\usepackage[T1]{fontenc}
\usepackage{mathrsfs}
\usepackage{latexsym}
\usepackage{amssymb}
\usepackage{graphicx}
\usepackage{float}
\usepackage{extarrows}
\usepackage{epstopdf}
\usepackage{caption}
\usepackage{subcaption}
\usepackage{amsmath}
\usepackage{wrapfig}
\usepackage[colorlinks]{hyperref}
\usepackage{lipsum}
\usepackage{subcaption}
\usepackage{multicol}
\usepackage{array}

\newtheorem{remark}{Remark}
\newtheorem{theorem}{Theorem}
\newtheorem{lemma}{Lemma}
\newtheorem{corollary}{Corollary}

\graphicspath{ {./Figure/}}

\journal{}

\begin{document}

\begin{frontmatter}



\title{A new symmetric linearly implicit exponential integrator preserving polynomial invariants or Lyapunov functions for conservative or dissipative systems}


\author[inst1]{Lu Li}

\affiliation[inst1]{organization={Machine Intelligence Department},
            addressline={Simula Metropolitan Center for Digital Engineering}, 
            city={Oslo},
            postcode={0167}, 
            country={Norway}}


\begin{abstract}
We present a new symmetric linearly implicit exponential integrator that preserves the polynomial first integrals or the Lyapunov functions for the conservative and dissipative stiff equations, respectively. The method is tested by both oscillated ordinary differential equations and partial differential equations, e.g., an averaged system in wind-induced oscillation,  the Fermi–Pasta–Ulam systems, and the polynomial pendulum oscillators. The numerical simulations confirm the conservative properties of the proposed method and demonstrate its good behavior in superior running
speed when compared with fully implicit schemes for long-time simulations.

\end{abstract}




\end{frontmatter}


\section{Introduction}
This paper focuses on the semilinear systems of the form 
\begin{equation}\label{stiff equation}
\dot{y}(t)=Ay+f(y),\quad y(t_0)=y_{0},
\end{equation} 
where $A$ is a linear unbounded differential operator or a matrix which has  eigenvalues with large negative real part or with purely imaginary eigenvalues of large modulus, and the non-linear term $f$ is supposed to be nonstiff satisfying Lipschitz condition. The semilinear system \eqref{stiff equation} arises in many applications, such as the charged-particle dynamic \cite{Cary}, the  rapidly rotating shallow water equations for semi-geostrophic particle motion \cite{Cotter}, and the semi-discretization of semilinear PDEs.
Equation \eqref{stiff equation} is usually stiff and
one popular class of  numerical integrators that are suitable for such problems is  the exponential integrators \cite{Hochbruck2010}. These methods  normally permit larger step sizes and provide higher accuracy than the non-exponential integrators. The basic idea behind such methods is to solve the stiff part with an exact solver. 
Based on the variation of constants formula,  the exact solution of \eqref{stiff equation} is given by 
\begin{equation}\label{exact-stiff equation}
y(t_0+h)=\text{exp}(hA)y_0+h\int_0^1\text{exp}((1-\tau)hA)f(y(t_0+\tau h))d\tau,
\end{equation}
where the integration interval is $[t_0,t_0+h]$.
Most exponential integrators can be obtained from an appropriate approximation of the integral in exact solution \eqref{exact-stiff equation}. For example,
the exponential Euler method is obtained by interpolating the
nonlinear term at $y_0$ with the form
\begin{equation*}\label{EEM}
y_{1}=\text{exp}(hA)y_0+h\phi(hA)f(y_0),
\end{equation*} 
and the implicit exponential Euler method is given by interpolating the
nonlinear term at  $y_{1}$ with the form
\begin{equation*}\label{IEEM}
y_{1}=\text{exp}(hA)y_0+h\phi(hA)f(y_{1}),
\end{equation*} 
where $\phi(z):=\frac{e^z-1}{z}$ \cite{Hochbruck2010}.
More examples of  exponential integrators can be found in, e.g., \cite{Hochbruck2010, Hochbruck2008}.

Equation \eqref{stiff equation} might possess  important geometric structures. In particular, the canonical Hamiltonian structure corresponds to  equations of the form
\begin{equation}\label{Hamilonian stiff equation}
\dot{y}(t)=J\nabla H(y),\quad y(t_0)=y_{0},
\end{equation} 
where 
$$H(y)=\frac{1}{2}y^TMy+ U(y),$$
and
\begin{equation*}\label{Canonical skew-symmetric matrix}
J = \left[ \begin{matrix}
0 & I_{m}\\
-I_{m} &  0 
\end{matrix} \right].
\end{equation*} 
with  $I_{m}$ an identity matrix, $M$ a symmetric matrix and $U(y)$ a scalar function. Two prominent features of equation \eqref{Hamilonian stiff equation} are the conservation of the energy function $H(y)$ and the preservation of the symplecticity. 
In this work,  we intend to focus on more general equations than the canonical Hamiltonian systems, where the matrix  $J$ in \eqref{Hamilonian stiff equation} is allowed to be a constant skew-symmetric matrix or a negative semidefinite matrix. A skew-symmetric matrix $J$ grantees the conservation of  the energy, while a negative semidefinite matrix $J$ will lead to a dissipative system with Lyapunov function $H(y)$  monotonically decreasing.  In view of these special structures in equation \eqref{Hamilonian stiff equation}, a type of candidate  methods will be the structure-preserving exponential integrators which
have superior qualitative behavior over long-time integration compared with the general-purpose designed higher-order methods \cite{Hairer2006}. Examples include symmetric methods \cite{Celledoni2008}, symplectic methods \cite{Wuxinyuan2018} and energy-preserving methods \cite{Wuxinyuan2016}.  Here we would like to consider particularly energy-preserving exponential integrators. Such methods in previous work are fully implicit, e.g., \cite{Wuxinyuan2016,miyatake2014energy,cui2021mass,shen2019geometric}, except for the recent work \cite{jiang2020linearly}, where a linearly implicit method for the nonlinear Klein--Gordon equation was considered using the scalar auxiliary variable (SAV) \cite{shen2019new} approach. For linearly implicit methods, only one linear system is solved in each iteration and thus less computational cost is needed.
For this consideration, we expect to construct linearly implicit methods in this work.

There are two mostly used techniques in creating linearly implicit energy-preserving methods for general conservative/dissipative systems with gradient flow, according to the authors' best knowledge. The first one follows from Furihata and co-authors, where multiple-point methods are used so that the nonlinearity can be portioned out \cite{furihata2011discrete}.  Further studies of this method were presented in  \cite{dahlby2011general} and \cite{eidnes2020linearly} using the concept of polarized energy and polarized discrete gradient. 
The second technique is to combine the linearly implicit Crank-Nicolson method and the invariant energy quadratization (IEQ) \cite{zhao2017numerical} or the scalar auxiliary variable (SAV) \cite{shen2019new} approach. 
Both the IEQ-based and the SAV-based methods are applicable for nonlinear systems, including nonpolynomials; however, these methods need bounded free energy regarding the nonlinear terms. Besides, the linearly implicit methods constructed based on the second approach have no symmetric property. This paper will focus on the technique using polarized energy for problems with polynomial energy functions. There are several reasons for choosing this technique. First, there are huge amounts of PDEs and ODEs with polynomial energy functions, e.g., the nonlinear Schrödinger equations, the nonlinear wave equations, the KdV equations, the Camassa-Holm equations, the wind-induced oscillator, 
the polynomial pendulum oscillator, and so on. Second, there are no other restrictions for the nonlinear terms except for being polynomials. Third, the general scheme of the methods based on the first technique looks much more straightforward than the second technique. Last but not least, the linearly implicit methods constructed based on the first technique can be symmetric, and it has been shown that symmetric methods applied to (near-)integrable reversible systems share similar properties to symplectic methods: linear error growth, long-time near-conservation of first integrals and existence of invariant tori \cite{hairer2006geometric}. Thus methods with symmetric property usually provide prominent long-time behavior.

The paper is organized as follows. First, we construct the symmetric linearly implicit energy-preserving exponential integrators and discuss their properties in Section 2. In Section 3, numerical examples are presented to illustrate the performance of the proposed method. In the last section,  we conclude the paper with a summary of the properties and advantages shared by the  method.


\section{Symmetric linearly implicit energy-preserving exponential scheme} 
In this section, we combine the idea of constructing linearly implicit methods using polarized energy \cite{dahlby2011general}  and the idea of constructing energy-preserving exponential methods using discrete gradient \cite{Wuxinyuan2016} to build the symmetric linearly implicit energy-preserving exponential integrators.  To present the method more intuitively, we restrict the nonlinear term $U(y)$ (potential energy) in equation \eqref{Hamilonian stiff equation} to be a cubic polynomial. However, the method is also applicable to problems with any higher-order polynomials, for which the results will be introduced briefly in this paper too. 

The critical point of using polarized energy to construct a linearly implicit method is to portion out the nonlinearity over consecutive time steps. This can be carried out by constructing quadratic polarized energy and then performing the polarized discrete gradient method, similarly as shown in \cite{eidnes2019linearly} for a cubic polynomial. A systematical way of constructing a quadratic polarization for higher-order polynomial functions is presented in  \cite{dahlby2011general}, for example 
\begin{itemize}
    \item  $U(x)=x^2$  can be polarized by  $\bar{U}(x,y)=\theta\frac{x^2+y^2}{2}+(1-\theta)xy$, $\theta\in [0,1]$,
    \item   $U(x)=x^3$  can be polarized by  $\bar{U}(x,y)=x\frac{x+y}{2}y,$
      \item   $U(x)=x^4$ can be polarized by  $\bar{U}(x,y)=x^2y^2,$
       \item  $U(x)=x^5$ can be polarized by  $\bar{U}(x,y,z,w)=xyzw\frac{x+y+z+w}{4}$,
    \item $U(x)=x^6$ can be polarized by  $\bar{U}(x,y,z)=x^2y^2z^2$.
\end{itemize}
Following from \cite{eidnes2019linearly},  $\overline{\nabla}\bar{U}$ is said to be a polarized discrete gradient for a polarized  energy $\bar{U}$ if  the following conditions hold
\begin{equation}\label{discrete eq}
\begin{split}
\bar{U}(y,z) - \bar{U}(x,y) &= 
\frac{1}{2}(z-x)^T \overline{\nabla}\bar{U}(x,y,z), \\
\overline{\nabla}\bar{U}(x,x,x)&= \nabla U(x). 
\end{split}
\end{equation}
We will use the polarized discrete gradient to construct the linearly implicit energy-preserving exponential integrators. Consider the variation of constants formula on interval $[t_0,t_0+2h]$ for problem \eqref{stiff equation}  in the form 
\begin{equation}\label{Variation of constants 2step}
y(t_0+2h)=\text{exp}(2hJM)y_0+2h\int_{0}^{1}\text{exp}(2(1-\xi)hJM)J\nabla U(y(t_0+2h\xi))d\xi.
\end{equation} 
Substituting $\nabla U(y(t_0+2h\xi))$ with the polarized discrete gradient $\overline{\nabla}\bar{U}(y_0,y_1,y_2)$ in \eqref{Variation of constants 2step}, we obtain the energy-preserving exponential integrator for problem \eqref{Hamilonian stiff equation} as follows
\begin{equation}\label{approx EP exponential integrator}
y_{n+2}=\text{exp}(2hJM)y_n+2h\phi(2hJM)J \overline{\nabla}\bar{U}(y_n,y_{n+1},y_{n+2}) \quad\quad \text{(LIEEP)}.
\end{equation}

\begin{remark}
Suppose that a quadratic polarization of a higher-order polynomial function $U(y_n)$  has the form  $\bar{U}(y_{n},\cdots,y_{n+i},\cdots y_{n+p-1})$. Then the generalization of the polarized discrete gradient in \eqref{discrete eq} can be given by 
\begin{equation}\label{discrete eq_higherorder}
\begin{split}
\bar{U}(y_{n+1},\cdots y_{n+p}) -\bar{U}(y_{n},\cdots y_{n+p-1})&= 
\frac{1}{p}( y_{n+p}-y_{n})^T \overline{\nabla}\bar{U}(y_{n},\cdots, y_{n+p}), \\
\overline{\nabla}\bar{U}(y_{n},\cdots, y_{n+p})&= \nabla U(y_n).
\end{split}
\end{equation}
\end{remark} 

\begin{remark}\label{orderp-scheme}
Suppose that the polarized discrete gradient for a higher-order polynomial energy function $U(y_n)$ is given by   $\overline{\nabla}\bar{U}(y_{n},y_{n+1},\cdots, y_{n+p})$ following from equation \eqref{discrete eq_higherorder}. Then the linearly implicit energy-preserving exponential integrator for problem  \eqref{Hamilonian stiff equation}  can be  given by 
\begin{equation}\label{approx EP exponential integrator_higherorder}
y_{n+p}=\text{exp}(phJM)y_n+ph\phi(phJM)J \overline{\nabla}\bar{U}(y_{n},y_{n+1},\cdots, y_{n+p}).
\end{equation} 
\end{remark} 
\begin{remark}
Scheme \eqref{approx EP exponential integrator} is  a special case of scheme \eqref{approx EP exponential integrator_higherorder} with $p=2$.
\end{remark} 

Before presenting the main theorems about the proposed method's conservative properties, we begin with a lemma.

\begin{lemma}\label{lemma}
For any symmetric matrix $M$, positive integer $p$ and  scalar $h>0$, the matrix 
\begin{equation*}\label{lemmaEq}
B=\text{exp}(phJM)^TM\text{exp}(phJM)-M
\end{equation*} 
is zero  when $J$ is skew-symmetric and  negative semidefinite  when $J$ is negative semidefinite.

\end{lemma}
The proof of this lemma follows directly from Lemma 2.2 in \cite{Wuxinyuan2016} since the result does not change when  $h$ is replaced by $ph$.

\begin{theorem}\label{enerPre-cubic}
Scheme \eqref{approx EP exponential integrator} preserves the following polarized energy 
\begin{equation}\label{Penergy}
\bar{H}(y_n,y_{n+1})=\frac{1}{4}(y_n^TMy_n+y_{n+1}^TMy_{n+1})+\bar{U}(y_n,y_{n+1}).
\end{equation} 
\end{theorem}

\begin{proof}\label{proof_PEP}
We first assume that the matrix $M$ is not singular and denote by  $V=2hJM$, $M^{-1}\overline{\nabla}\bar{U}=\overline{\nabla}\tilde{U}$. The energy error has the form
\begin{equation}\label{energyer}
\begin{split}
&\bar{H}(y_{n+1},y_{n+2})-\bar{H}(y_n,y_{n+1})\\
&=\frac{1}{4}(y_{n+2}^TMy_{n+2}+y_{n+1}^TMy_{n+1})-\frac{1}{4}(y_n^TMy_n+y_{n+1}^TMy_{n+1})\\
&+\bar{U}(y_{n+1},y_{n+2})-\bar{U}(y_n,y_{n+1}).\\
\end{split}
\end{equation} 
By replacing $y_{n+2}=\text{exp}(V)y_n+2h\phi(V)J \overline{\nabla}\bar{U}(y_n,y_{n+1},y_{n+2})$ and using $\phi(V)V=\text{exp}(V)-I$, we get the following equations
\begin{equation}\label{linearEER}
\begin{split}
&\frac{1}{4}(y_{n+2}^TMy_{n+2}+y_{n+1}^TMy_{n+1})-\frac{1}{4}(y_n^TMy_n+y_{n+1}^TMy_{n+1})\\
&=\frac{1}{4}y_n^T(\text{exp}(V)^TM\text{exp}(V)-M)y_n+hy_n^T\text{exp}(V)^TM\phi(V)J\overline{\nabla}\bar{U}\\
&+h^2(\overline{\nabla}\bar{U})^TJ^T\phi(V)^TM\phi(V)J\overline{\nabla}\bar{U}\\
&=\frac{1}{4}y_n^T(\text{exp}(V)^TM\text{exp}(V)-M)y_n+\frac{1}{2}y_n^T\text{exp}(V)^TM(\text{exp}(V)-I)\overline{\nabla}\tilde{U}\\
&+\frac{1}{4}(\overline{\nabla}\tilde{U})^T(\text{exp}(V)-I)^TM(\text{exp}(V)-I)\overline{\nabla}\tilde{U},
\end{split}
\end{equation} 
and
\begin{equation}\label{nonlinearEER}
\begin{split}
&\bar{U}(y_{n+1},y_{n+2})-\bar{U}(y_n,y_{n+1})\\
&=\frac{(y_{n+2}-y_n)^T}{2}\overline{\nabla}\bar{U}(y_n,y_{n+1},y_{n+2})\\
&=\frac{1}{2}y_n^T(\text{exp}(V)^T-I)\overline{\nabla}\bar{U}+h(\overline{\nabla}\bar{U})^TJ^T\phi(V)^T\overline{\nabla}\bar{U}\\
&=\frac{1}{2}y_n^T(\text{exp}(V)^TM-M)\overline{\nabla}\tilde{U}+\frac{1}{2}\overline{\nabla}\tilde{U}^TV^T\phi(V)^TM\overline{\nabla}\tilde{U}\\
&=\frac{1}{2}y_n^T(\text{exp}(V)^TM-M)\overline{\nabla}\tilde{U}+\frac{1}{2}\overline{\nabla}\tilde{U}^T(\text{exp}(V)^TM-M)\overline{\nabla}\tilde{U}.
\end{split}
\end{equation} 
Inserting equations \eqref{linearEER} and \eqref{nonlinearEER} to equation \eqref{energyer}, we obtain the following results
\begin{equation}\label{energyer final}
\begin{split}
&\bar{H}(y_{n+1},y_{n+2})-\bar{H}(y_n,y_{n+1})\\
&=\frac{1}{4}y_n^T(\text{exp}(V)^TM\text{exp}(V)-M)y_n+\frac{1}{2}y_n^T(\text{exp}(V)^TM\text{exp}(V)-M)\overline{\nabla}\tilde{U}\\
&+\frac{1}{4}\overline{\nabla}\tilde{U}^T(\text{exp}(V)^TM\text{exp}(V)-M)\overline{\nabla}\tilde{U}+\frac{1}{4}\overline{\nabla}\tilde{U}^T(\text{exp}(V)^TM-M\text{exp}(V))\overline{\nabla}\tilde{U}\\
&=\frac{1}{4}(y_n+\overline{\nabla}\tilde{U})^T(\text{exp}(V)^TM\text{exp}(V)-M)(y_n+\overline{\nabla}\tilde{U})\\
&+\frac{1}{4}\overline{\nabla}\tilde{U}^T(\text{exp}(V)^TM-M\text{exp}(V))\overline{\nabla}\tilde{U}\\
&=0,
\end{split}
\end{equation} 
where the last step follows from the fact that $\text{exp}(V)^TM-M\text{exp}(V)$ is skew-symmetric, and $\text{exp}(V)^TM\text{exp}(V)-M$ is also skew-symmetric from Lemma \ref{lemma}.

For a singular $M$, one can find a series of non-singular and symmetric matrices $M_\epsilon$ such that $M_\epsilon= M$ when $\epsilon\rightarrow 0$. For any $M_\epsilon$, we can follow the proof above and show that the polarized energy function in the form
\begin{equation*}\label{PH with epsilon}
\bar{H_\epsilon}(y^\epsilon_n,y^\epsilon_{n+1})=\frac{1}{4}({y^\epsilon_n}^TM_\epsilon y^\epsilon_n+{y^\epsilon_{n+1}}^TM_\epsilon y^\epsilon_{n+1})+\bar{U}(y^\epsilon_n,y^\epsilon_{n+1})
\end{equation*} 
is preserved by the approximation given by 
$$y^\epsilon_{n+2}=\text{exp}(2hJM_\epsilon)y_n^\epsilon+2h\phi(2hJM_\epsilon)J \overline{\nabla}\bar{U}(y_{n}^\epsilon,y_{n+1}^\epsilon,y_{n+2}^\epsilon)$$
for the following problem 
\begin{equation*}\label{}
\dot{y}(t)=JM_\epsilon y+J\nabla U(y(t)),\quad y(t_0)=y_{0}.
\end{equation*} 
Therefore, $\bar{H_\epsilon}(y^\epsilon_n,y^\epsilon_{n+1})=\bar{H}(y_n,y_{n+1})$ is preserved by method \eqref{approx EP exponential integrator} when $\epsilon\rightarrow 0$.

\end{proof}
For problems \eqref{Hamilonian stiff equation} with higher-order polynomial energy functions, similar conservation property can be obtained by scheme \eqref{approx EP exponential integrator_higherorder}, as shown in the following corollary.
\begin{corollary}\label{orerp-energy}
Scheme \eqref{approx EP exponential integrator_higherorder} preserves the following polarized energy 
\begin{equation}\label{polaEn higherorder}
\bar{H}(y_{n},\cdots y_{n+p-1})=\frac{1}{2p}\sum_{i=0}^{p-1}y_{n+i}^TMy_{n+i}+\bar{U}(y_{n},\cdots, y_{n+p-1}).
\end{equation} 
\end{corollary}
The proof is similar to Theorem \ref{enerPre-cubic} and thus is omitted here.
\begin{theorem}\label{enerdissip-cubic}
If $J$ is a constant negative semidefinite matrix, scheme \eqref{approx EP exponential integrator} preserves the polarized Lyapunov function $\bar{H}$ for problem \eqref{Hamilonian stiff equation}:
$$\bar{H}(y_{n+1},y_{n+2})\le \bar{H}(y_{n+1},y_{n}),$$
where $\bar{H}(y_n,y_{n+1})$ is defined  by equation \eqref{Penergy}.
\end{theorem}

\begin{proof}
Let us suppose $M$ to be non-singular; otherwise we follow the similar procedure in the proof for Theorem \ref{enerPre-cubic}, i.e., constructing a series of $M_\epsilon$ convergent to $M$ to achieve the result.

For a constant negative semidefinite matrix $J$, the error of Lyapnov function has the same form as the energy error in \eqref{energyer final}:
\begin{equation*}\label{Lyapunov final}
\begin{split}
&\bar{H}(y_{n+1},y_{n+2})-\bar{H}(y_n,y_{n+1})\\
&=\frac{1}{4}(y_{n+2}^TMy_{n+2}+y_{n+1}^TMy_{n+1})-\frac{1}{4}(y_n^TMy_n+y_{n+1}^TMy_{n+1})\\
&+\bar{U}(y_{n+1},y_{n+2})-\bar{U}(y_n,y_{n+1})\\
&=\frac{1}{4}(y_n+\overline{\nabla}\tilde{U})^T(\text{exp}(V)^TM\text{exp}(V)-M)(y_n+\overline{\nabla}\tilde{U})\\
&+\frac{1}{4}\overline{\nabla}\tilde{U}^T(\text{exp}(V)^TM-M\text{exp}(V))\overline{\nabla}\tilde{U}\\
&\le 0,
\end{split}
\end{equation*} 
where the last step follows from the fact that $\text{exp}(V)^TM-M\text{exp}(V)$ is skew-symmetric, and  $\text{exp}(V)^TM\text{exp}(V)-M$ is negative semidefinite according to Lemma \ref{lemma}.

\end{proof}
Similarly, for problems \eqref{Hamilonian stiff equation} with a higher-order polynomial  energy, we have the following corollary.
\begin{corollary}
If $J$ is a constant negative semidefinite matrix, scheme \eqref{approx EP exponential integrator_higherorder} preserves the polarized Lyapunov function:
$$\bar{H}(y_{n+1},\cdots, y_{n+p})\le \bar{H}(y_{n},\cdots, y_{n+p-1}),$$
where $\bar{H}(y_{n},\cdots y_{n+p-1})$ has the same form as equation \eqref{polaEn higherorder}.
\end{corollary}
The proof is similar to Theorem \ref{enerdissip-cubic} and thus is omitted here.

\begin{theorem}\label{LIEEP_symmetric}
Scheme \eqref{approx EP exponential integrator} is symmetric.
\end{theorem}\label{proof-LIEEP_symmetric}
\begin{proof}
Exchanging $y_{n}$, $y_{n+1}$ $\leftrightarrow$ $y_{n+2} $, $ y_{n+1}$ and replacing $2h$ by $-2h$ in \eqref{approx EP exponential integrator}, we obtain
 \begin{equation*}\label{symmetric}
y_n=\text{exp}(-V)y_{n+2}-2h\phi(-V)J\overline{\nabla}\bar{U}(y_{n+2},y_{n+1},y_{n}).
\end{equation*} 
Following from the definition in \eqref{discrete eq} and the  cyclic permutation free property of the polarized energy \cite{dahlby2011general},
 we get 
 \begin{equation}\label{symmetric_PDG-}
\overline{\nabla}\bar{U}(y_{n+2},y_{n+1},y_{n})=\overline{\nabla}\bar{U}(y_{n},y_{n+1},y_{n+2}).
 \end{equation}
Using $\text{exp}(V)\phi(-V)=\phi(V)$ and the relation in \eqref{symmetric_PDG-}, we then obtain
 \begin{equation*}\label{symmetric}
y_{n+2}=\text{exp}(V)y_{n}+2h\phi(V)J\overline{\nabla}\bar{U}(y_{n},y_{n+1},y_{n+2}).
\end{equation*} 
\end{proof}

Scheme \eqref{approx EP exponential integrator_higherorder} for problems with higher-order polynomial $U(y)$ also holds the symmetric property if the polarization of the function $U(y)$ is invariant when the order of its arguments is reversed. This can be obtained  by symmetrizing over dihedral group \cite{dahlby2011general}. Although only cyclic permutation free is required in the definition of the polarized energy,  we can actually always get a permutation free quadratic polarization for any higher order polynomial $\bar{U}(y)$. In fact,  the polarization examples shown above are all permutation free. In this paper, we always consider permutation free polarization, i.e., quadratic polarization satisfying the following  condition
\begin{equation*}\label{permutation free}
\bar{U}(y_{n},\cdots y_{n+p-1})=\bar{U}(y_{n+i_1-1},\cdots y_{n+i_p-1}),\quad (i_1,\cdots,i_p)\in S_p,
\end{equation*} 
where $S_p$ is a symmetric group.

\begin{corollary}
Scheme \eqref{approx EP exponential integrator_higherorder} 
is symmetric if the polarization of function $U(y)$ is permutation free.
\end{corollary}
The proof is similarly to the proof of Theorem \ref{LIEEP_symmetric} except that the cyclic permutation free property should be replaced by the permutation free property.

\section{Numerical experiment}
The proposed method is suitable for conservative or dissipative differential equations of the form \eqref{Hamilonian stiff equation} with $J$ a constant skew-symmetric or negative semidefinite matrix and $U(y)$ a scalar polynomial function of any order. These equations include the highly oscillatory conservative or dissipative  ODEs and also the semi-discrete systems arising from PDEs.  In this section, we test our method by three differential equations. The first two examples are used to demonstrate the efficient behavior of the method compared to the fully implicit method, e.g., the energy-preserving exponential integrators based on the averaged vector field method, denoted by EAVF. The third example is chosen to show the application of the method for problems with higher-order polynomial energy functions. 

EAVF method was put forward in  \cite{Wuxinyuan2016}, which has the form
 \begin{equation}\label{EAVF}
y_{n+1}=\text{exp}(hJM)y_n+h\phi(hJM)J \overline{\nabla}U(y_n,y_{n+1}), \quad\quad 
\end{equation} 
where $\overline{\nabla}U(y_n,y_{n+1})=\int_{0}^1 \nabla U((1-\tau)y_n+\tau y_{n+1})d\tau$.
Besides, scheme \eqref{EAVF} has been shown to preserve the discrete energy of the form
\begin{equation}\label{EAVF-energy}
H(y_n)=\frac{1}{2}y_n^TMy_n+U(y_n).
\end{equation} 
The integral in EAVF method is evaluated by the 2-point GL quadrature formula, which gives an exact approximation of the integration. In most cases, the terms $\text{exp}(phJM)$ and $\phi(phJM)$ ($p$ is a positive integer number) can not be calculated explicitly, and  we use the
MATLAB package proposed in \cite{berland2007expint} to compute them, where Pade approximations are used.  

In the experiments, the global error is defined by 
\begin{equation*}\label{gler}
\underset{n\ge 0}{\max}\|y_n-y(t_n)\|, 
\end{equation*} 
where $t_n=t_0+nh$ with $h$ the time step size, and $y(t_n)$ is the reference exact solution.
In this work, we compute the reference solution by  the 6-order continuous Runge–Kutta (CRK) method \cite{hairer2010energy} with the form
\begin{equation*}\label{CRK6GL5}
\begin{cases}
      y_{n+1/3}=y_n+hJ\int_0^1\big(\frac{37}{27}-\frac{32}{9}\sigma+\frac{20}{9}\sigma^2\big)\nabla H(Y_\sigma)d\sigma\\
      \vspace{5pt}
     y_{n+2/3}=y_n+hJ\int_0^1\big(\frac{26}{27}+\frac{8}{9}\sigma-\frac{20}{9}\sigma^2\big)\nabla H(Y_\sigma)d\sigma\\
     y_{n+1}=y_n+hJ\int_0^1\nabla H(Y_\sigma)d\sigma
    \end{cases},
\end{equation*} 
where
 \begin{equation*}
 \begin{split}
Y_\sigma=&-\frac{(3\sigma-1)(3\sigma-2)(\sigma-1)}{2}y_n+\frac{3\sigma(3\sigma-2)(3\sigma-3)}{2}y_{n+1/3}\\
&-\frac{3\sigma(3\sigma-1)(3\sigma-3)}{2}y_{n+2/3}+\frac{\sigma(3\sigma-1)(3\sigma-2)}{2}y_{n+1},
 \end{split}
\end{equation*} 
and the integrals are evaluated exactly by the 5-point GL quadrature.  
For all fully implicit schemes, we solve the nonlinear system   by  the fixed point iteration with tolerance as $10^{-14}$. All the numerical results presented are obtained from schemes implemented in MATLAB (2020a release), running on a MacBook Pro
with a dual-core 2.6 GHz Intel 6-Core i7 processor and 16 GB of 2667 MHz DDR4
RAM. 

\vspace{2pt}

\noindent\textbf{Test problem one}.  We consider an averaged system in wind-induced oscillation \cite{mclachlan1998unified}
\begin{equation}\label{problem1}
\begin{split}
\dot{x}_1&=-\zeta x_1-\lambda x_2+x_1x_2,\\
\dot{x}_2&=\lambda x_1-\zeta x_2+\frac{1}{2}(x_1^2-x_2^2),
\end{split}
\end{equation} 
where $\zeta\ge 0$ is a damping factor and $\lambda$ is a detuning parameter with $\zeta=r\text{cos}(\theta)$, $\lambda=r\text{sin}(\theta)$, $r\ge 0$, $0\le\theta\le\pi/2$.
Equation \eqref{problem1} can be rewritten into the form \eqref{Hamilonian stiff equation} with 
\begin{equation*}\label{problem1 J}
J=\begin{bmatrix}
    -cos(\theta)      & -sin(\theta)   \\
    sin(\theta)      & -cos(\theta) 
\end{bmatrix}
\quad
M=\begin{bmatrix}
    r      & 0   \\
    0      & r
\end{bmatrix},
\end{equation*} 
\begin{equation*}
U=-\frac{1}{2}sin(\theta)(x_1x_2^2-\frac{1}{3}x_1^3)+\frac{1}{2}cos(\theta)(\frac{1}{3}x_2^3-x_1^2x_2).
\end{equation*}

Its energy function (when $\theta=\pi/2$) or Lyapunov function (dissipative case, when $\theta\le\pi/2$) is
\begin{equation*}\label{problem1 H}
H=\frac{1}{2}r(x_1^2+x_2^2)-\frac{1}{2}sin(\theta)(x_1x_2^2-\frac{1}{3}x_1^3)+\frac{1}{2}cos(\theta)(\frac{1}{3}x_2^3-x_1^2x_2).
\end{equation*} 
The matrix exponential in scheme \eqref{approx EP exponential integrator} for problem \eqref{problem1} can be calculated explicitly as follows
\begin{equation*}\label{exp(V)}
exp(V)=\begin{bmatrix}
    exp(-2hcr)cos(2hsr)      & -exp(-2hcr)sin(2hsr)   \\
    exp(-2hcr)sin(2hsr)      & exp(-2hcr)cos(2hsr)
\end{bmatrix},
\end{equation*} 
with $c=cos(\theta)$ and $s=sin(\theta)$.
We can obtain a polarized discrete gradient $\nabla\bar{U}(x^n,x^{n+1},x^{n+2})$ based on a polarization of $U$ given by 
 \begin{equation}\label{polarisation of U}
 \begin{split}
\bar{U}(x^n,x^{n+1})=&-\frac{1}{2}sin(\theta)\big(a\frac{x_1^n+x_1^{n+1}}{2}x_2^nx_2^{n+1}+(1-a)\frac{x_1^n(x_2^{n+1})^2+x_1^{n+1}(x_2^{n})^2}{2}\\
&-\frac{1}{3}x_1^n\frac{x_1^n+x_1^{n+1}}{2}x_1^{n+1}\big)+\frac{1}{2}cos(\theta)\big(\frac{1}{3}x_2^n\frac{x_2^n+x_2^{n+1}}{2}x_2^{n+1}\\
&-ax_1^nx_1^{n+1}\frac{x_2^n+x_2^{n+1}}{2}
-(1-a)\frac{x_2^n(x_1^{n+1})^2+x_2^{n+1}(x_1^{n})^2}{2}\big).
 \end{split}
\end{equation} 
Then we get the linearly implicit energy-preserving scheme in the form of \eqref{approx EP exponential integrator} and the polarized energy in the form of \eqref{Penergy}.

Consider the initial vector $x_1(0)=0, \quad x_2(0)=1$, $r=20$, step size $h=1/20$, and the parameters $\theta=\pi/2$ or $\theta=\pi/2-10^{-4}$. For  LIEEP method, the starting point $[x_1^1, x_2^1]$ is computed by the Matlab function \emph{ode45}. $\theta=\pi/2$ provides a conservative system, and Figure  \ref{energy upwind-conserve} confirms that EAVF method preserves the discrete energy \eqref{EAVF-energy} and LIEEP method preserves the polarized energy \eqref{Penergy}. While  $\theta=\pi/2-10^{-4}$ leads to a dissipative system, and Figure  \ref{energy upwind-dissipative} shows that EAVF method and  LIEEP method preserves the dissipation of the Lyapunov function in \eqref{EAVF-energy} and \eqref{Penergy}, respectively. 

\begin{figure}[H]
\centering
      \begin{subfigure}[b]{0.45\textwidth}
      \centering
                \includegraphics[width=0.99\textwidth]{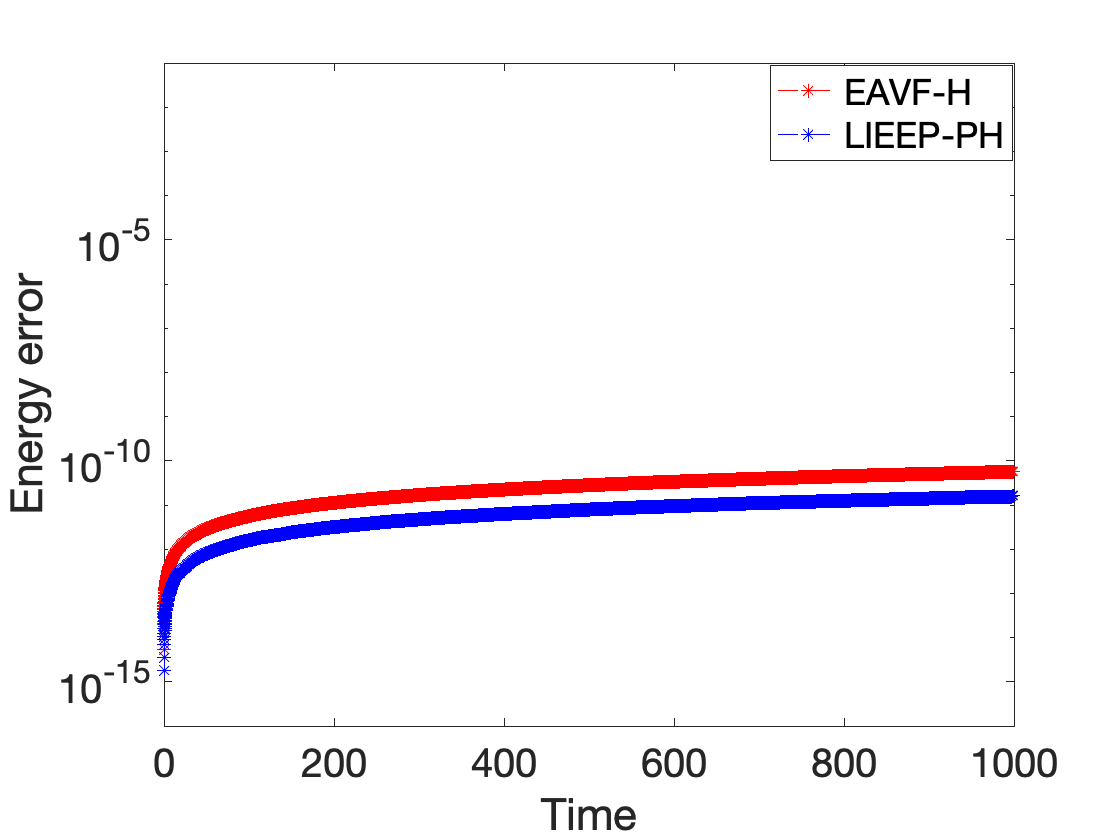}
                \caption{$\theta=\pi/2$}
      \label{energy upwind-conserve}
        \end{subfigure}\hspace{18pt}
          \begin{subfigure}[b]{0.45\textwidth}
        \centering
                \includegraphics[width=0.99\textwidth]{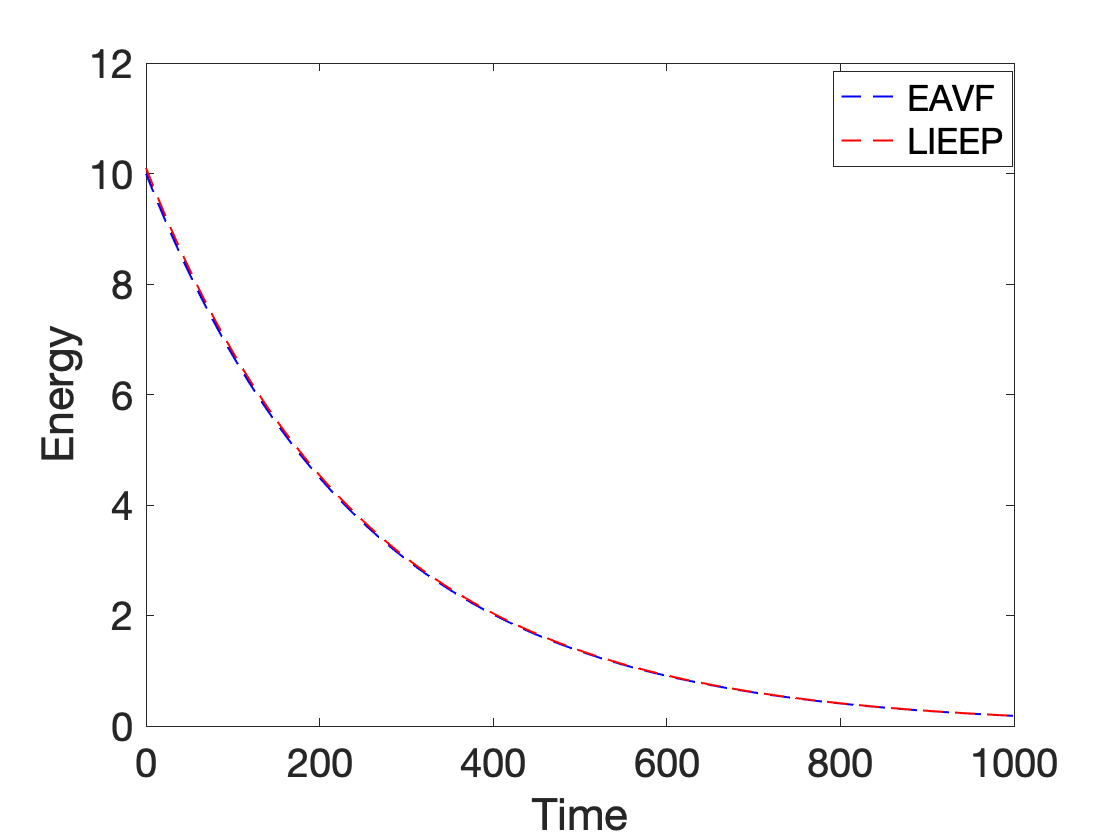} \caption{$\theta=\pi/2-10^{-4}$}
      \label{energy upwind-dissipative}
        \end{subfigure}
      \caption{ The energy behaviour of both  EAVF and  LIEEP method with time step size $h=1/20$. }
      \label{energy upwind}
\end{figure}

In Figure \ref{energy upwind cons efficiency order conserv} and  \ref{energy upwind cons efficiency order dissip}, we consider the global errors and the computational cost using step sizes  $h=1/10/2^{i}$, with $i=0,1,\cdots, 5$. Surprisingly, Figure  \ref{orderplot upwind-conserve} shows that LIEEP method is superconvergent for the conservative system ($\theta=\pi/2$). We find that this behavior is closely related to  the parameter $a$ in the polarized potential energy in \eqref{polarisation of U}. We have tried $a=0, 1/2,1/4, 1$, but only $a=1/2$ gives a three-order behaviour. Figure  \ref{efficiencyplot upwind-conserve} shows that the proposed method is more efficient than the fully implicit EAVF method.  When $\theta=\pi/2-10^{-4}$, i.e., the system is dissipative,  the superconvergent behavior disappears for LIEEP method, see  Figure \ref{orderplot upwind-dissipative}. From this figure, we observe that LIEEP method has a convergent issue when the  step size is $h=1/10$, but with the decrease of the time step size, LIEEP method gets convergent and behaves even better than  EAVF method.  Figure \ref{efficiencyplot upwind-dissipative} indicates that the proposed method is much more efficient than EAVF method for the dissipative system.

\begin{figure}[H]
\centering
      \begin{subfigure}[b]{0.45\textwidth}
      \centering
                \includegraphics[width=0.99\textwidth]{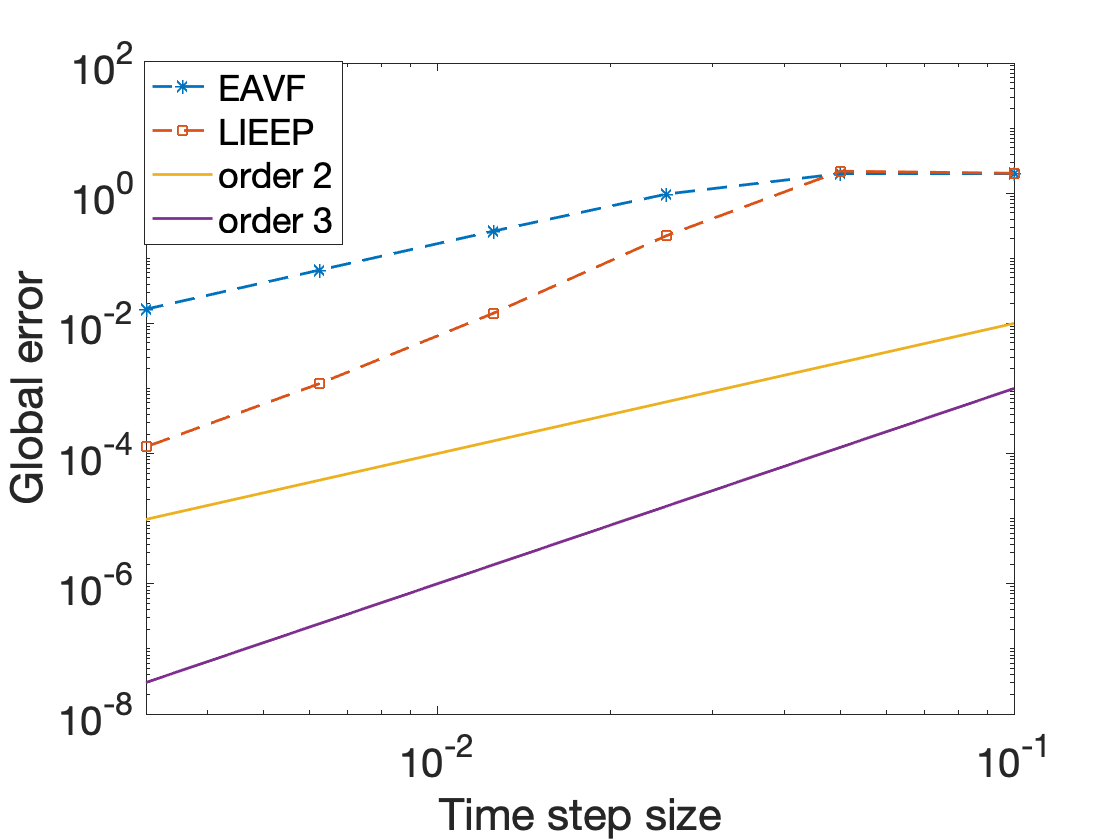}
                 \caption{order plot}
      \label{orderplot upwind-conserve}
        \end{subfigure}\hspace{18pt}
          \begin{subfigure}[b]{0.45\textwidth}
        \centering
                \includegraphics[width=0.99\textwidth]{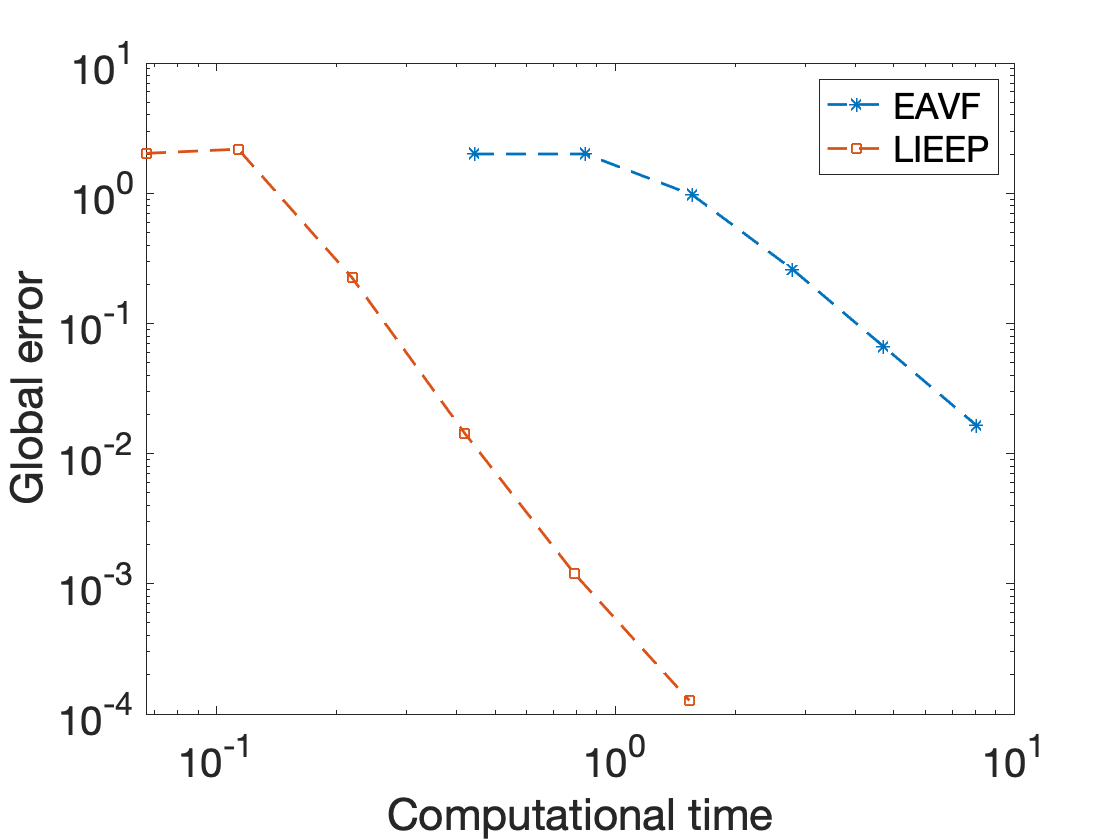}
                  \caption{efficiency plot}
      \label{efficiencyplot upwind-conserve}
        \end{subfigure}
      \caption{$T=1000$, $\theta=\pi/2$, the time step sizes are  $h=1/10/2^{i}$, for $i=0,1,\cdots, 5$.}
      \label{energy upwind cons efficiency order conserv}
\end{figure}

\begin{figure}[H]
\centering
      \begin{subfigure}[b]{0.45\textwidth}
      \centering
                \includegraphics[width=0.99\textwidth]{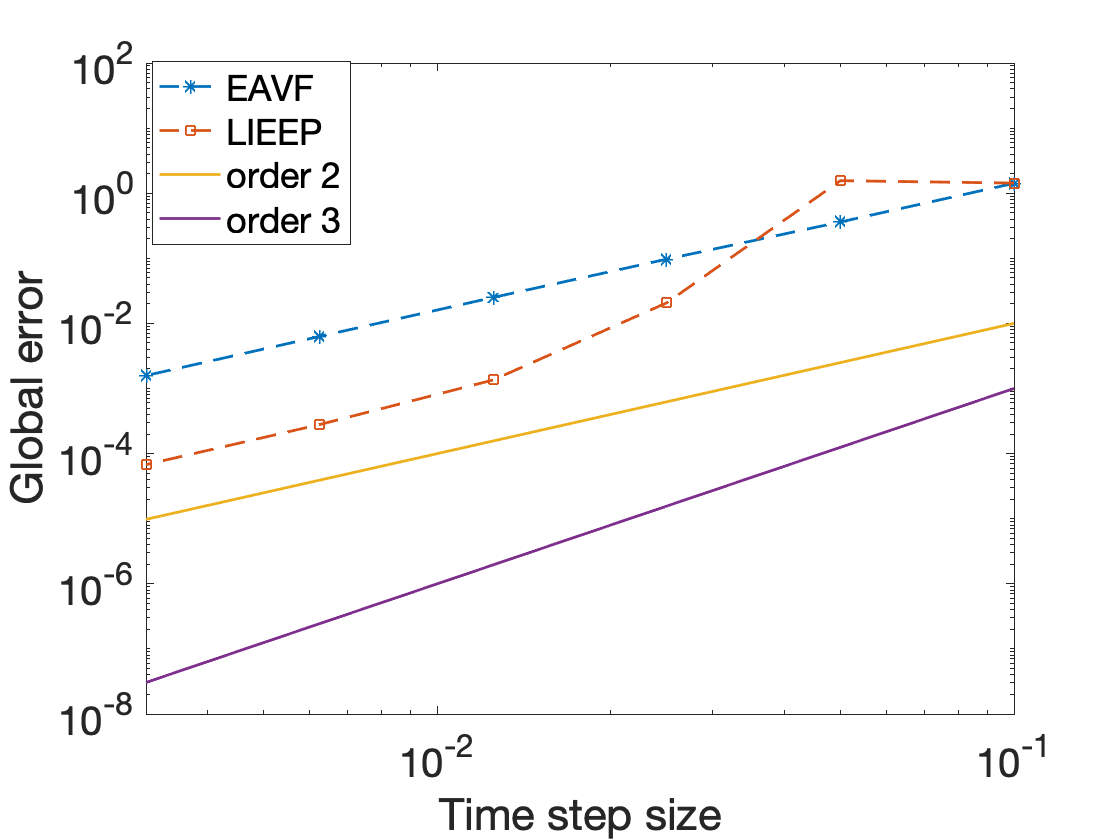}
                  \caption{order plot}
      \label{orderplot upwind-dissipative}
        \end{subfigure}\hspace{18pt}
          \begin{subfigure}[b]{0.45\textwidth}
        \centering
                \includegraphics[width=0.99\textwidth]{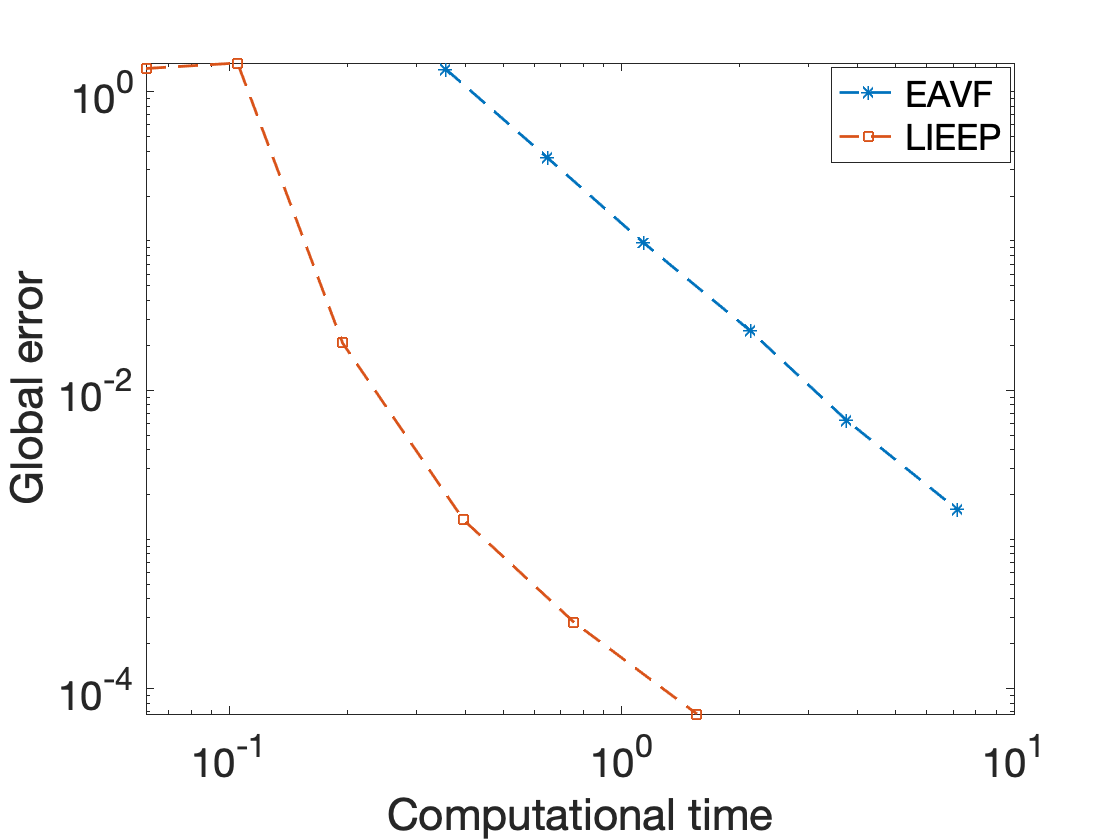}
                  \caption{efficiency plot}
      \label{efficiencyplot upwind-dissipative}
        \end{subfigure}
      \caption{$T=1000$, $\theta=\pi/2-10^{-4}$, the time step sizes are  $h=1/10/2^{i}$, for $i=0,1,\cdots, 5$.}
      \label{energy upwind cons efficiency order dissip}
\end{figure}

\noindent\textbf{Test problem two.}
 We consider  a continuous generalization of an $\alpha$-FPU (Fermi-Pasta-Ulam) system \cite{macias2009implicit}:
\begin{equation}\label{problem2}
\frac{\partial^2 u}{\partial t^2}=\beta\frac{\partial^3 u}{\partial t\partial x^2}+\frac{\partial^2 u}{\partial x^2}(1+\epsilon(\frac{\partial u}{\partial x})^p)-\gamma\frac{\partial u}{\partial t}-m^2u,
\end{equation} 
where $\epsilon>0$, $\beta\ge 0$ is the coefficient of the internal damping,  $ \gamma\ge 0$ is the coefficient of the external damping, and $(x,t)\in[0, L]\times[0, T]$. Taking $\frac{\partial u}{\partial t} =v$, equation \eqref{problem2} can be rewritten as 
\begin{equation}\label{problem2-rewritten}
\begin{split}
\partial_t u&=v\\
\partial_t v&=\beta\frac{\partial^2 v}{\partial x^2}+\frac{\partial^2 u}{\partial x^2}(1+\epsilon(\frac{\partial u}{\partial x})^p)-\gamma v-m^2u.
\end{split}
\end{equation} 
Denoting by $y=[u,v]^T$, equation \eqref{problem2-rewritten} can be reformulated as the following Hamiltonian form
\begin{equation*}\label{problem1-rewritten-Hamiltonian pde}
\frac{\partial y}{\partial t}=\mathcal{Q}\frac{\delta \mathcal{H}}{\delta y},
\end{equation*} 
where 
\begin{equation*}\label{Hamiltonian}
\mathcal{Q}=\begin{bmatrix}
    0& 1 \\
    -1 & \beta\partial^2_x-\gamma
  \end{bmatrix},\quad
\mathcal{H}=\int_0^L E(t, u, v, u_x)dx,
\end{equation*} 
with 
\begin{equation*}\label{local energy}
E(t, u, v, u_x)=\frac{1}{2}u_x^2+\frac{m^2}{2}u^2+\frac{1}{2}v^2+\epsilon\frac{u_x^{p+2}}{(p+2)(p+1)}.
\end{equation*} 
The function $E(t, u, v, u_x)$ physically represents the local energy density of  system \eqref{problem2} at any time $t$. 

Consider $p=1$ and the homogeneous Dirichlet boundary conditions $u(0,t)=u(L,t)=0$. Discretizing $\partial_x^2$ with the central difference operator and $\partial_x$ with the forward difference operator,   we obtain the following semi-discrete  ODE system

\begin{equation*}\label{problem1 semi-ODE}
\dot{y}=Q(My+\nabla U(y)),
\end{equation*} 
where
\begin{equation*}\label{problem1 J}
Q=\begin{bmatrix}
    0      & I  \\
    -I      & \beta D- \gamma I
\end{bmatrix},
\quad
M=\begin{bmatrix}
    m^2I-D      & 0   \\
    0      & I
\end{bmatrix},
\quad
U(y)=\sum_{j=0}^{N-1}\frac{\epsilon}{6}(\frac{u_{j+1}-u_j}{\Delta x})^3.
\end{equation*} 
Setting $w_j^n=\frac{u_{j+1}^n-u_j^n}{\Delta x}$,
and defining the polarized energy
\begin{equation*}\label{problem PFU polarised energy}
\bar{U}(w_j^n, w_j^{n+1})=\sum_{j=0}^{N-1}\frac{\epsilon}{6}w_j^n\frac{w_j^n+w_j^{n+1}}{2}w_j^{n+1},
\end{equation*} 
we can obtain the polarized discrete gradient
\begin{equation*}\label{PFU polarised discrete gradient}
\begin{split}
\bar{\nabla}\bar{U}(w_j^n, w_j^{n+1}, w_j^{n+2} )
&=\frac{\epsilon}{6\Delta x}w_{j-1}^{n+1}
(w_{j-1}^{n}+w_{j-1}^{n+1}+w_{j-1}^{n+2})\\
&-\frac{\epsilon}{6\Delta x}w_{j}^{n+1}(w_{j}^{n}+w_{j}^{n+1}+w_{j}^{n+2}),
\end{split}
\end{equation*}
and the discrete gradient
\begin{equation*}\label{PFU  discrete gradient}
\bar{\nabla}U(w_j^n )
=\frac{\epsilon}{2\Delta x}(w_{j-1}^{n})^2-
\frac{\epsilon}{2\Delta x}(w_{j}^{n})^2.
\end{equation*} 

We consider $m=0$, $\epsilon=\frac{3}{4}$,  $L=128, \quad T=100$ and spatial step size $\Delta x=1$.  The initial conditions are set to be $u_j(0)=q_j(0)$, $v_j(0)=\dot{q}_j(0)$ and 
{\small
\begin{equation*}\label{PFU  initial condition}
q_j(t)=
5\text{ln}\frac{1+\text{exp}\big(2(\alpha(j-97)+t\text{sinh}(\alpha))\big)}{1+\text{exp}\big(2(\alpha(j-96)+t\text{sinh}(\alpha))\big)}+5\text{ln}\frac{1+\text{exp}\big(2(\alpha(j-32)+t\text{sinh}(\alpha))\big)}{1+\text{exp}\big(2(\alpha(j-33)+t\text{sinh}(\alpha))\big)},
\end{equation*} 
}
where $\alpha=0.1$.
For LIEEP method, the starting point $y_1$ is computed by the 6-order CRK method.

In Figure \ref{FPU  energy betas gamma0}, we fix the external damping coefficient to be zero ($\gamma=0$) and present the energy behavior of LIEEP method for systems with different internal damping coefficients and a long simulation time $T=500$.  We observe that the numerical method preserves the energy when there is no damping ($\beta=0$) and also preserves the dissipation property when the internal damping coefficient is greater than zero, consistent with what is observed in \cite{macias2009implicit},  where a fully implicit four-step method is considered. Similar behavior is observed in  Figure \ref{FPU  energy beta0 gammas}, where the internal damping coefficient is set to be zero ($\beta=0$). Figure  \ref{FPU beta 0 order efficiency} and \ref{FPU beta 2 order efficiency} confirm that both EAVF  and  LIEEP method are of order 2 in time, and the comparison of the computational cost between these two methods gives a clear evidence that the proposed method is more efficient than  EAVF method. In this experiment, we also present the numerical solutions given by LIEEP method for systems with different settings of $\beta$ and $\gamma$, see Figure  \ref{FPU solution}. These figures clearly demonstrate the dissipative nature of the
external damping coefficient, see the change of the colors between Figure  \ref{FPU gamma 0 solution} and  Figure \ref{FPU gamma 0005 solution}, and the internal damping coefficient, see the change of the shapes between Figure \ref{FPU gamma 0 solution} and Figure \ref{FPU beta2 solution}. These observations in the numerical solutions are in accordance with the results shown by the fully implicit four-step method in \cite{macias2009implicit}.

\begin{figure}[H]
\centering
      \begin{subfigure}[b]{0.45\textwidth}
      \centering
                \includegraphics[width=0.99\textwidth]{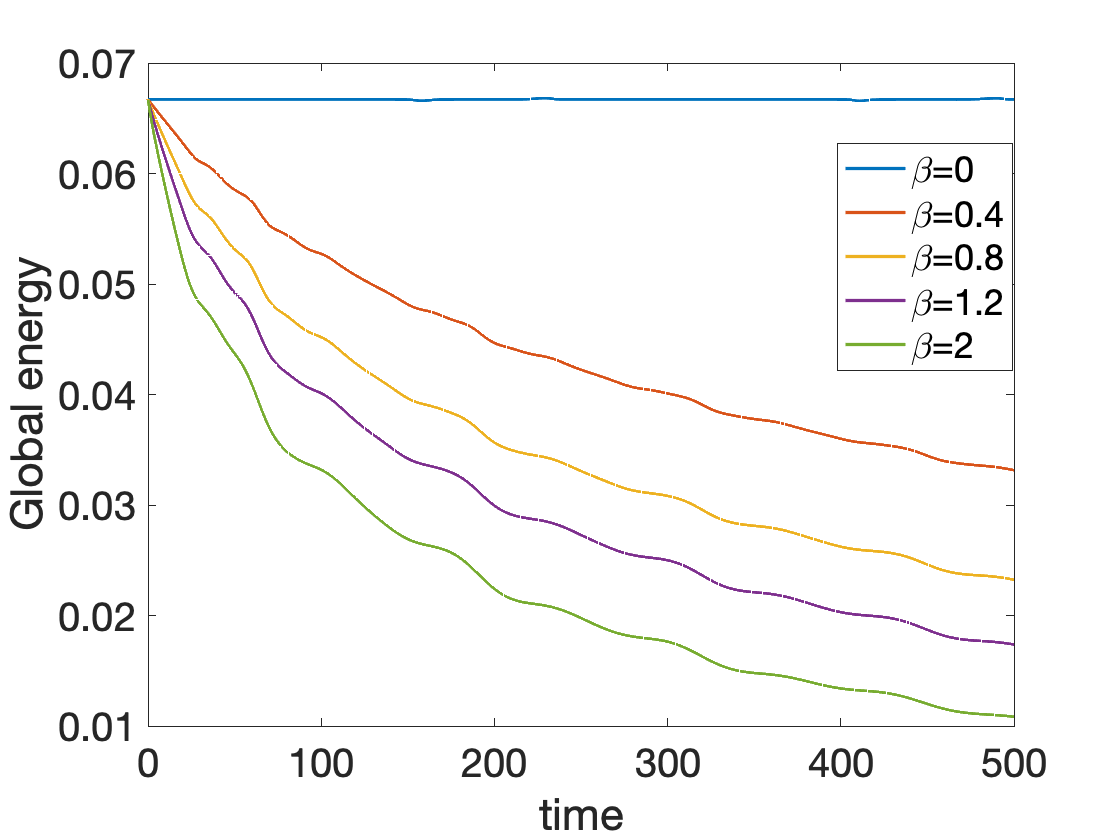}
                \caption{$\gamma=0$}
      \label{FPU  energy betas gamma0}
        \end{subfigure}\hspace{18pt}
          \begin{subfigure}[b]{0.45\textwidth}
        \centering
                \includegraphics[width=0.99\textwidth]{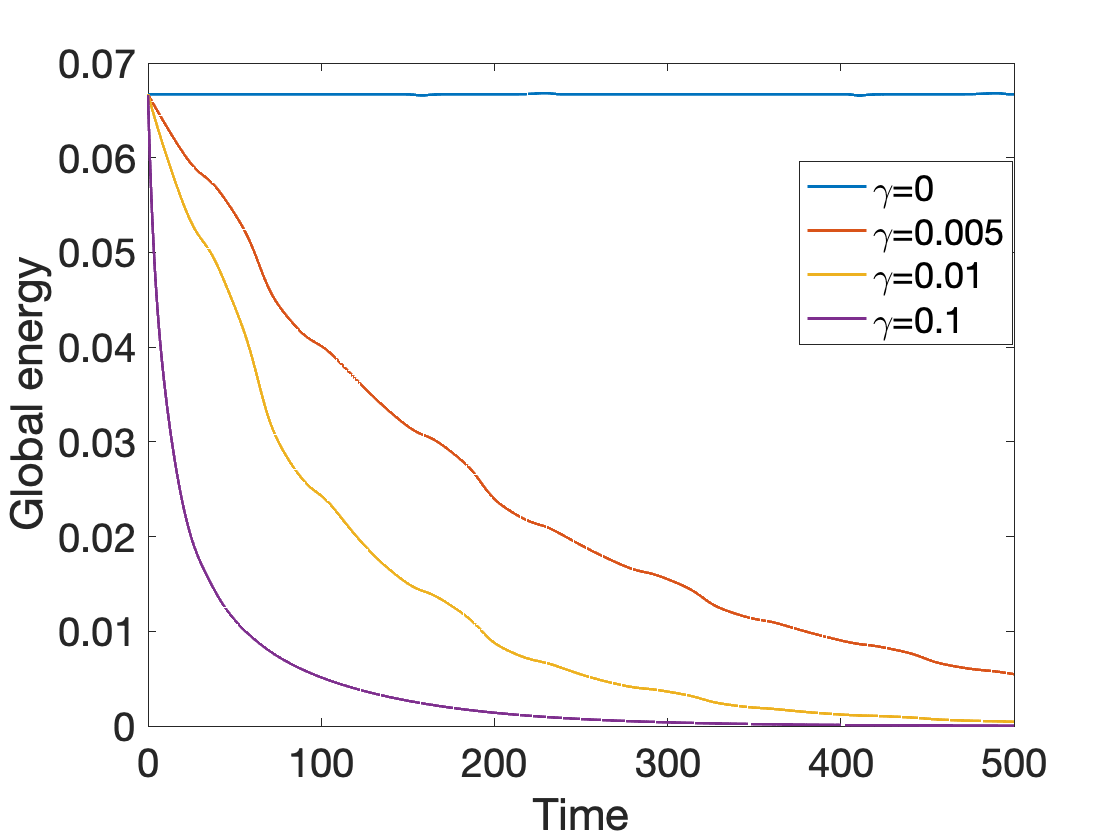}
                \caption{$\beta=0$. }
      \label{FPU  energy beta0 gammas}
        \end{subfigure}
      \caption{Polarized energy by LIEEP method for $\alpha$-FPU system with different settings of internal and external damping coefficients. $T=500$ and time step size $h=0.025$.}
      \label{FPU  energy beta gamma}
\end{figure}

\begin{figure}[H]
\centering
      \begin{subfigure}[b]{0.45\textwidth}
      \centering
                \includegraphics[width=0.99\textwidth]{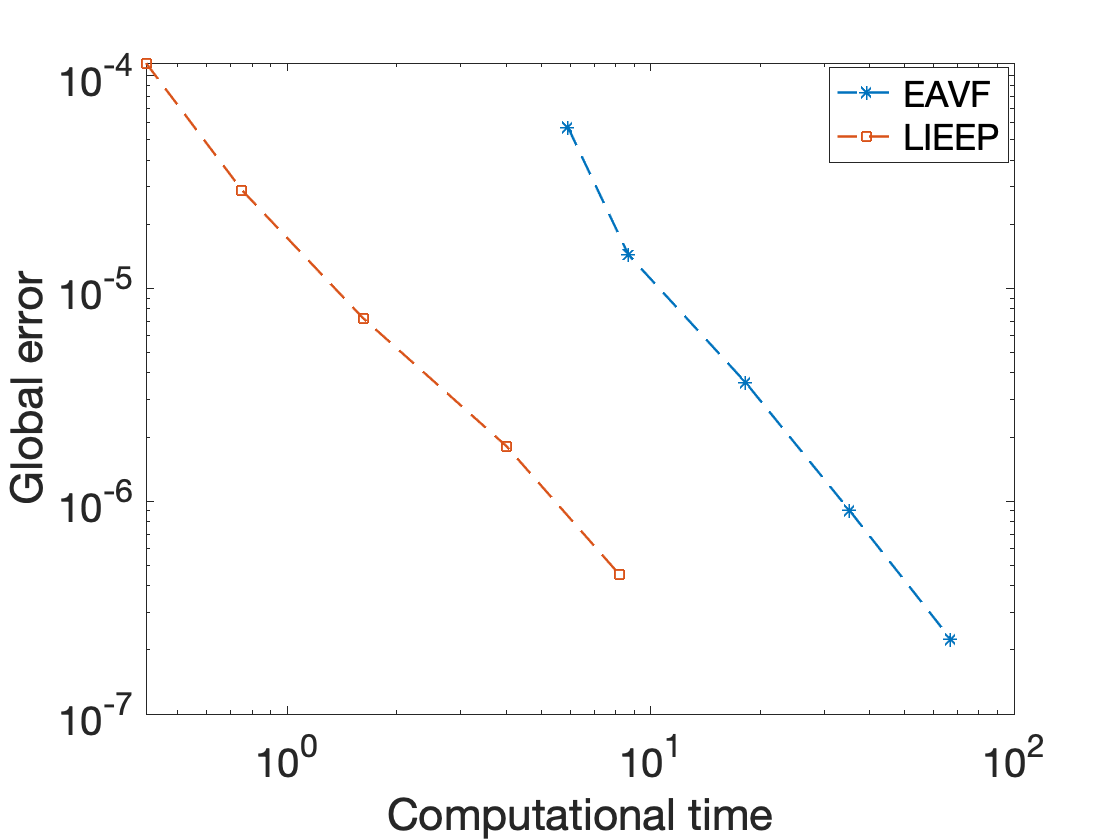}
        \end{subfigure}\hspace{18pt}
          \begin{subfigure}[b]{0.45\textwidth}
        \centering
                \includegraphics[width=0.99\textwidth]{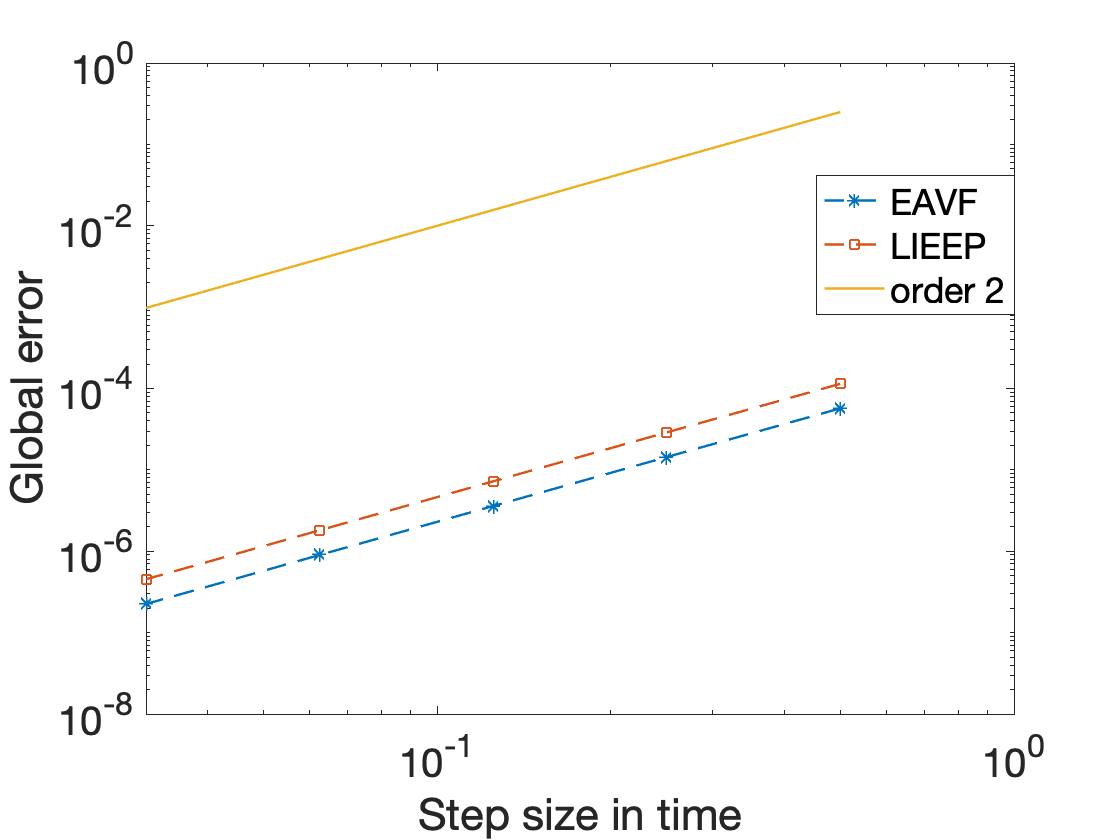}
        \end{subfigure}
      \caption{ $T=100$, $\gamma=0.005$, $\beta=0$ and time step size $h=\frac{1}{2^i}$, $i=1,\cdots,5$, \textit{Left:} efficiency comparison; \textit{right:} order plot. }
      \label{FPU beta 0 order efficiency}
\end{figure}

\begin{figure}[H]
\centering
      \begin{subfigure}[b]{0.45\textwidth}
      \centering
                \includegraphics[width=0.99\textwidth]{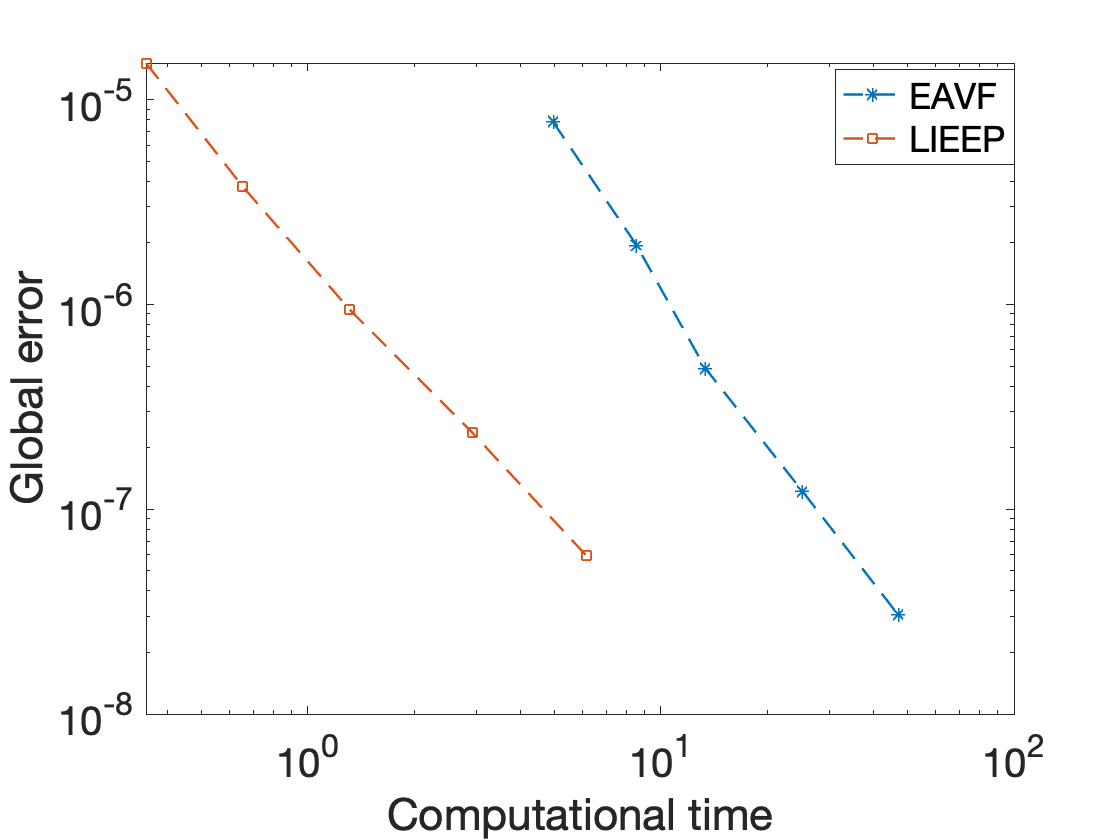}
        \end{subfigure}\hspace{18pt}
          \begin{subfigure}[b]{0.45\textwidth}
        \centering
                \includegraphics[width=0.99\textwidth]{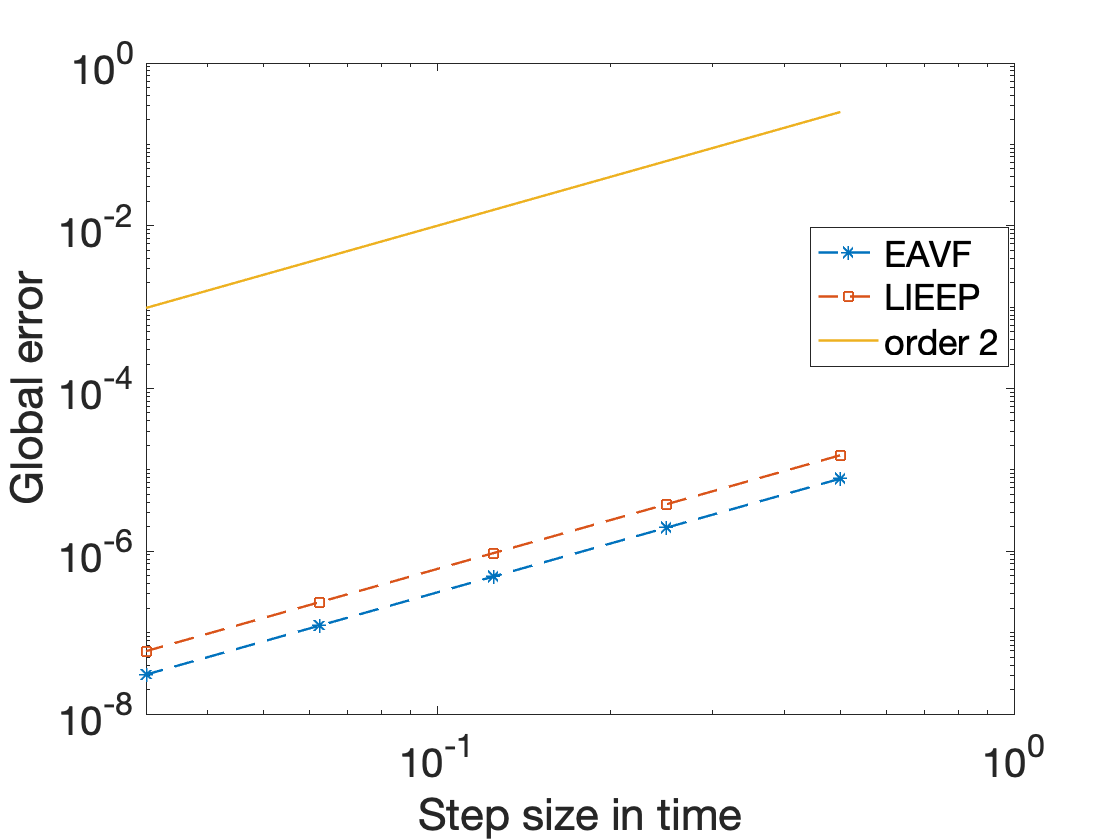}
        \end{subfigure}
      \caption{$T=100$, $\gamma=0$, $\beta=2$ and time step size $h=\frac{1}{2^i}$, $i=1,\cdots,5$. \textit{Left:} efficiency comparison; \textit{right:} order plot. }
      \label{FPU beta 2 order efficiency}
\end{figure}

\begin{figure}[H]
\centering
      \begin{subfigure}[b]{0.45\textwidth}
      \centering
                \includegraphics[width=0.99\textwidth]{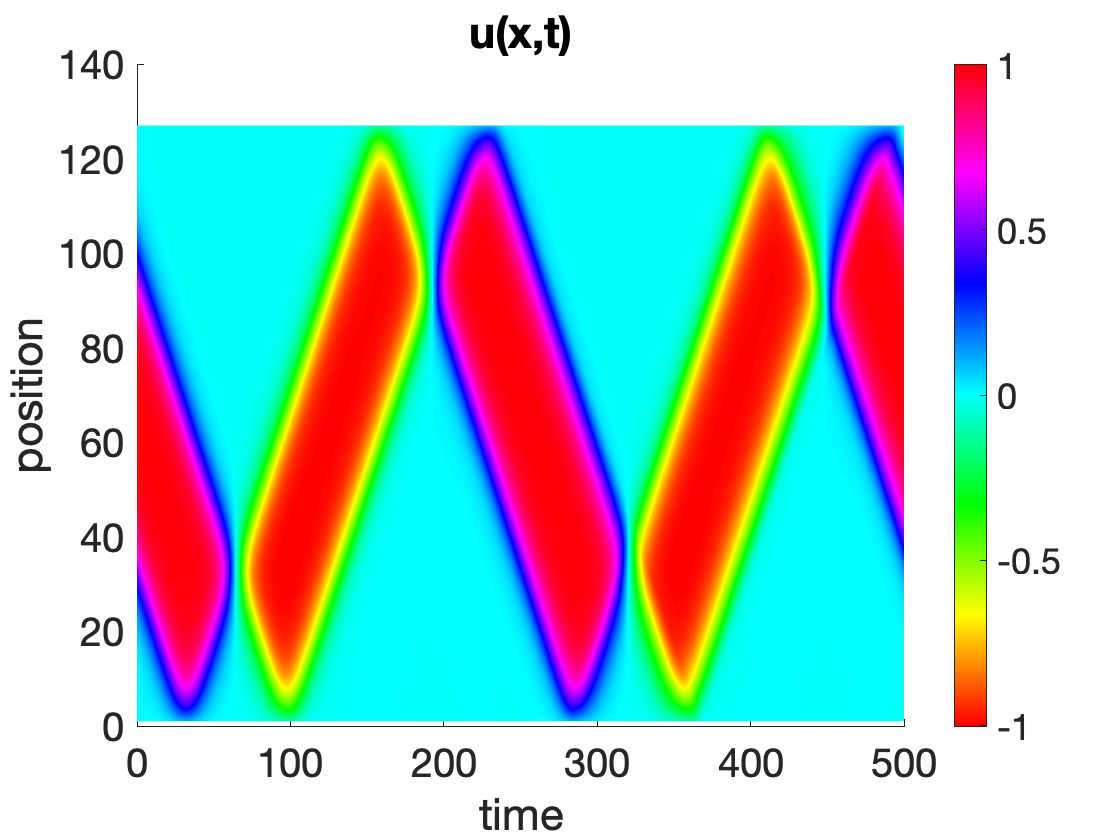}
                \caption{$\gamma=0$, $\beta=0$.}
      \label{FPU gamma 0 solution}
        \end{subfigure}
        
          \begin{subfigure}[b]{0.45\textwidth}
        \centering
                \includegraphics[width=0.99\textwidth]{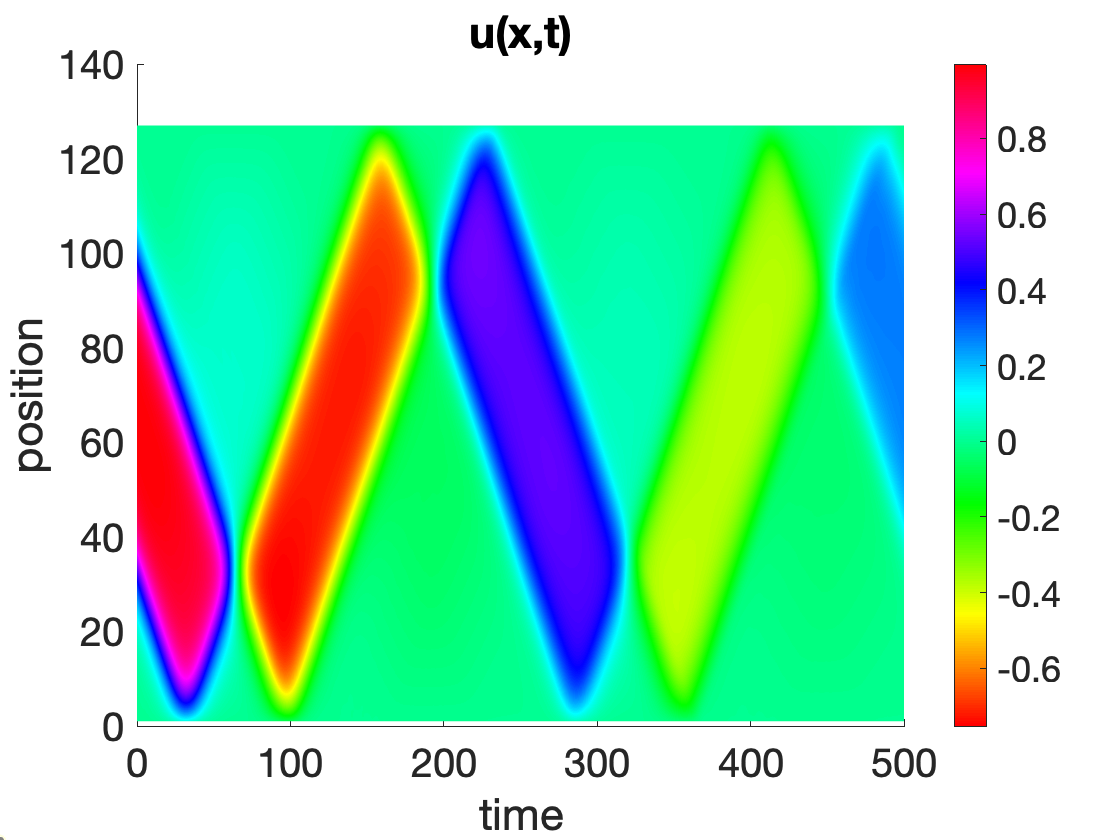}
                  \caption{$\gamma=0.005$, $\beta=0$.}
      \label{FPU gamma 0005 solution}
        \end{subfigure}
          \begin{subfigure}[b]{0.45\textwidth}
        \centering
                \includegraphics[width=0.99\textwidth]{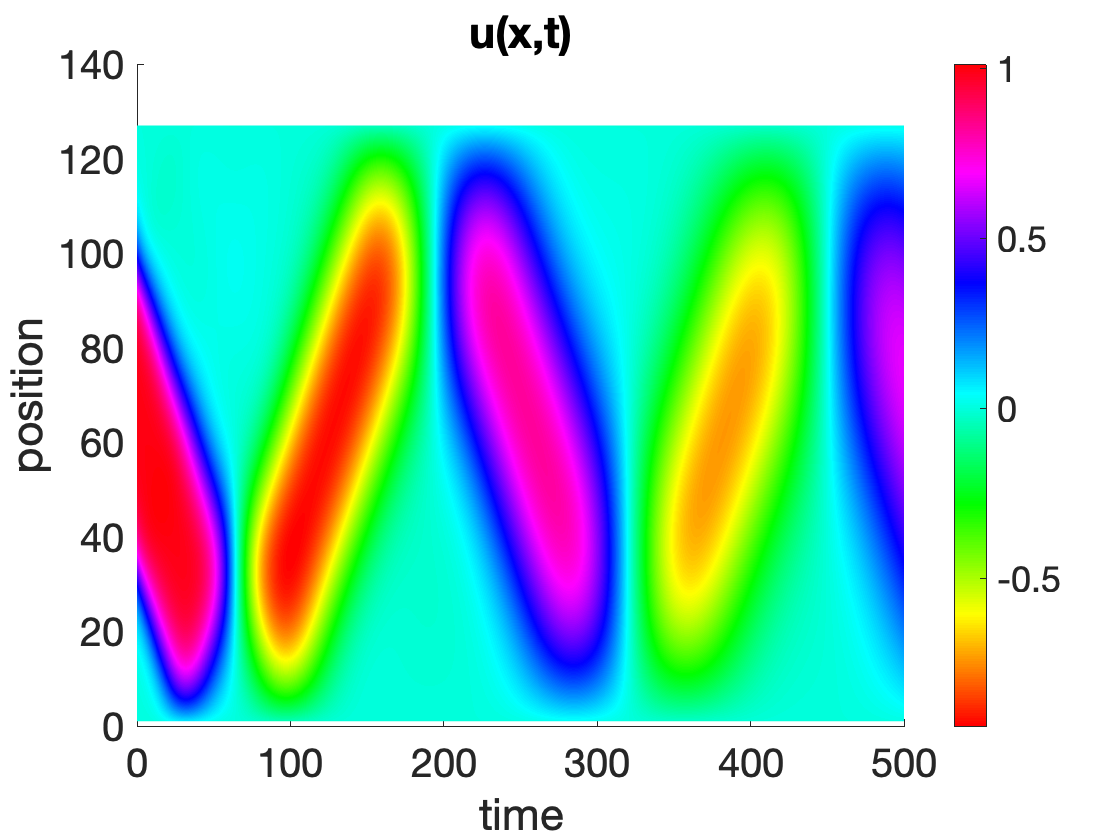}
                 \caption{$\gamma=0$, $\beta=2$.}
      \label{FPU beta2 solution}
        \end{subfigure}
      \caption{ The numerical solution of the  $\alpha$-FPU system with different settings of internal and external damping coefficients on time interval $[0, 500]$ and with time step size $h=0.025$.   }
      \label{FPU solution}
\end{figure}

\vspace{2pt}

\noindent\textbf{Test problem three.}
We consider the polynomial pendulum oscillator, and the main focus of this example is to illustrate the energy conservation property shown in Corollary \ref{orerp-energy} for the proposed method in Remark \ref{orderp-scheme}. Consider the nonlinear pendulum problem with the Hamiltonian 
\begin{equation*}\label{pendulum energy}
H(p,q)=\frac{1}{2}p^2+1-\text{cos}q,
\end{equation*} 
and a truncated Taylor expansion of the cosine function:
\begin{equation}\label{pendulum Taylor}
H(p,q)=\frac{1}{2}p^2+\frac{1}{2}q^2-\frac{1}{24}q^4+\frac{1}{720}q^6.
\end{equation} 
The approximation in \eqref{pendulum Taylor} to the original problem will be more accurate if even higher-order polynomial is used and   $\lvert q\rvert$ is sufficiently small, e.g., $\lvert q\rvert<\frac{1}{2}$ \cite{iavernaro2009high}. 
Denoting by $y=[q;p]$, the polynomial pendulum oscillator with energy function \eqref{pendulum Taylor} can be rewritten into form \eqref{Hamilonian stiff equation} with $m=2$, $J$ the canonical skew-symmetric matrix, M the identity matrix and 
\begin{equation}\label{pendulum poly U}
U(y)=-\frac{1}{24}q^4+\frac{1}{720}q^6.
\end{equation} 
Consider a polarization of \eqref{pendulum poly U} as follows
\begin{equation}\label{polaried pendulum poly U}
\bar{U}(y_n,y_{n+1},y_{n+2})=-\frac{1}{24}q_nq_{n+1}q_{n+2}\frac{q_n+q_{n+1}+q_{n+2}}{3}+\frac{1}{720}{q_n}^2{q_{n+1}}^2{q_{n+2}}^2.
\end{equation} 
We can obtain a polarized discrete gradient of the form
\begin{equation*}\label{polarised gradient U}
\begin{split}
\bar{\nabla}\bar{U}(y_n,y_{n+1},y_{n+2},y_{n+3} )=&\frac{1}{240}{q_{n+1}}^2{q_{n+2}}^2(q_{n}+q_{n+3})\\
&-\frac{1}{24}q_{n+1}q_{n+2}(q_n+q_{n+1}+q_{n+2}+q_{n+3}).
\end{split}
\end{equation*} 
Take the initial value as  $q_0=0.5$, $p_0=1$ and the integration interval as $[0,1000]$. We compute the first two starting points $y_1$ and $y_2$ by  Matlab function \emph{ode15s}. The polarized energy is reported in Figure \ref{pendulum energy fig}, and we observe that it is exactly preserved by LIEEP method defined by equation \eqref{approx EP exponential integrator_higherorder}. In this figure, we also present the original discrete energy by LIEEP method , i.e.,
\begin{equation}\label{discrete pendulum energy}
H(p_n,q_n)=\frac{1}{2}p_n^2+1-\text{cos}q_n.
\end{equation} 
Although LIEEP method does not preserve the exact original energy, Figure \ref{pendulum energy fig} shows that the discrete energy given by  LIEEP method in the form of \eqref{discrete pendulum energy}  stays oscillated and bounded over a long-time integration. Besides, we observe  that the numerical solution by LIEEP method applied to the truncated equation provides an approximation with a similar behaviour as the exact solution of the nonlinear pendulum oscillator if  a small time step size is considered, e.g., h=0.3, i.e.,  the phase space is a cylinder,  as illustrated in Figure \ref{pendulum solution}.

\begin{figure}[H]
\centering
      \begin{subfigure}[b]{0.45\textwidth}
      \centering
                \includegraphics[width=0.99\textwidth]{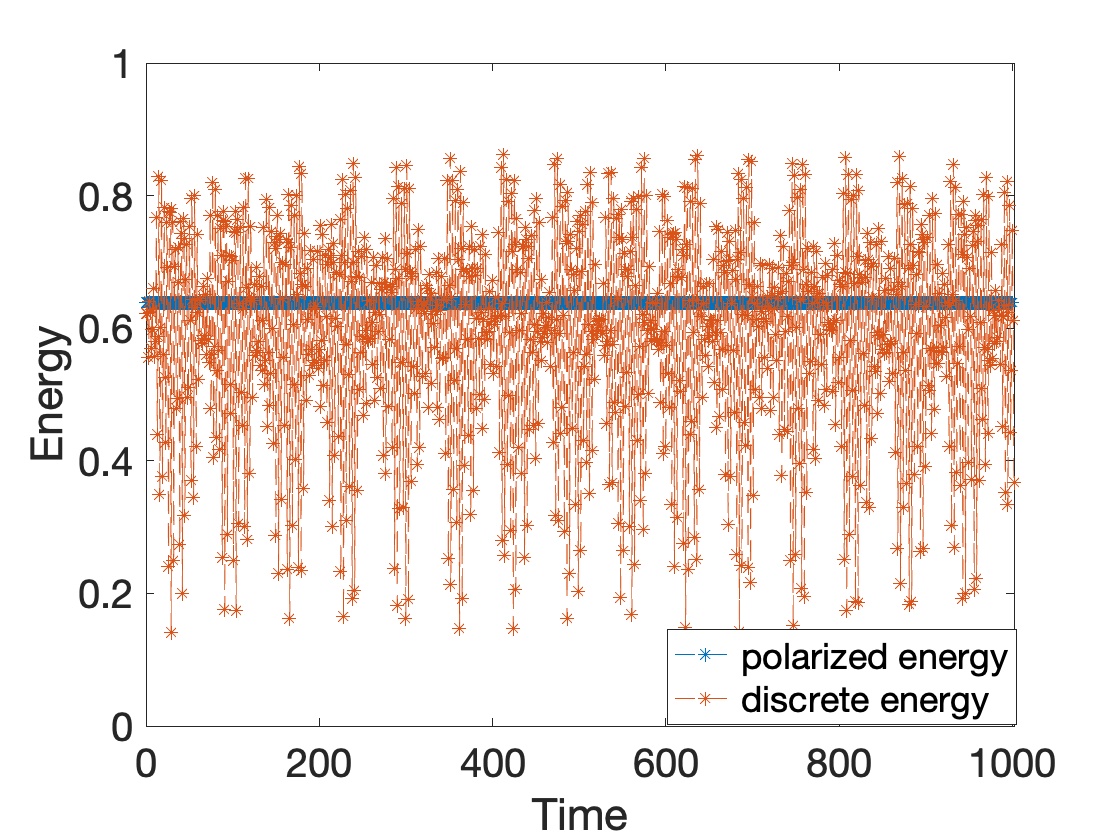}
                 \caption{Energy} \label{pendulum energy fig}
        \end{subfigure}\hspace{18pt}
          \begin{subfigure}[b]{0.45\textwidth}
        \centering
                \includegraphics[width=0.99\textwidth]{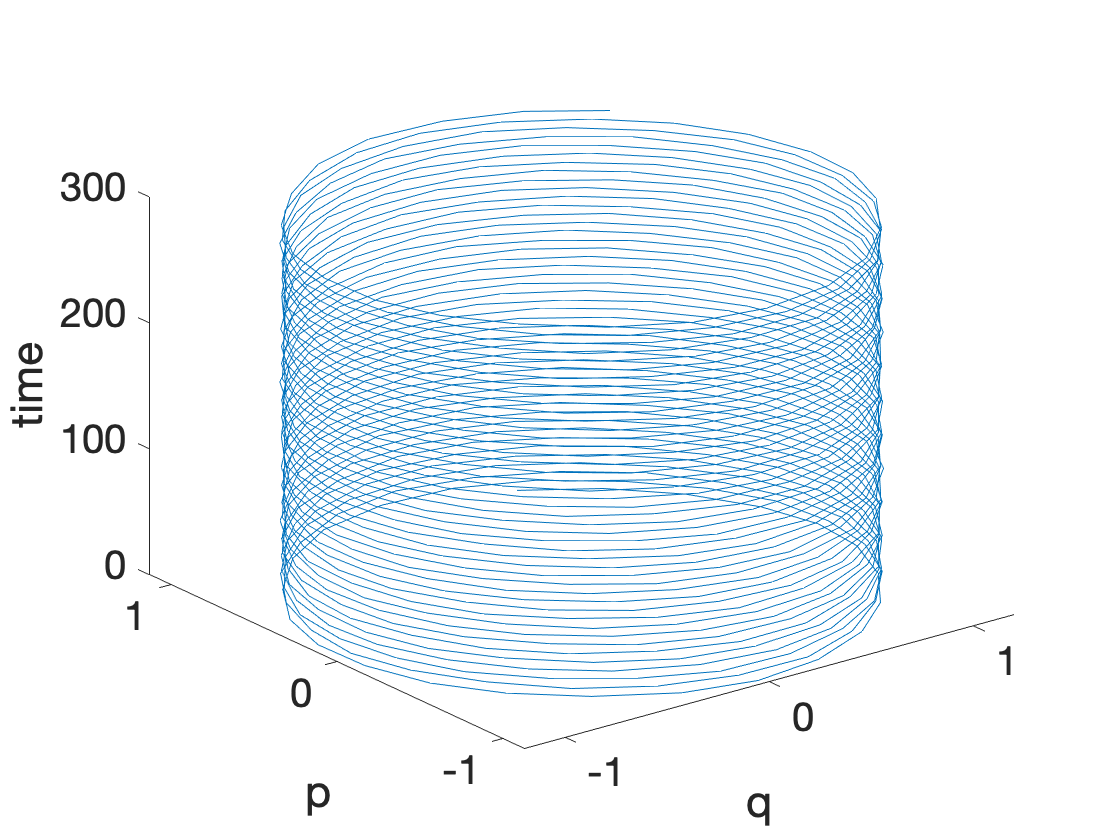}
                 \caption{solution}\label{pendulum solution}
        \end{subfigure}
      \caption{ In \ref{pendulum energy fig}, the time step size is $h=1$, the polarized energy is defined by equation  \eqref{polaEn higherorder} with $p=3$ and $\bar{U}(y_n,y_{n+1},y_{n+2})$ defined by \eqref{polaried pendulum poly U}; the discrete energy is defined by equation \eqref{discrete pendulum energy}. In \ref{pendulum solution}, the time step size is $h=0.3$.}
      \label{pendulum}
\end{figure}

\section{Conclusion}
This paper constructs a novel symmetric linearly implicit  exponential integrator that holds the conservative properties for semi-linear problems  with polynomial energy functions. The method is developed based on combining the idea of using polarized discrete gradient to build linearly implicit methods and the idea of using discrete gradient to create energy-preserving exponential integrators.
Besides conservative properties, the method is shown to be symmetric, which guarantees excellent long-time behavior.


We test our methods on three types of differential equations, including an oscillated ODE, i.e., the averaged wind-induced oscillator, an oscillated PDE, i.e.,  the damped FPU problem, and also an ODE with higher-order polynomial energy fucntion, i.e., the polynomial pendulum oscillator. The numerical experiments confirm that the proposed method preserves the polarized energy or the Lyapunov function, and the method is of order two. Moreover, it has been shown that the proposed method has superconvergent behavior for some particular systems when a proper polarized energy is considered. 
Compared with the fully implicit method (EVAF), our method shows a significantly lower computational cost. In view of the nice properties and the good behavior, we recommend the proposed method for problems with polynomial energy function.



\section*{Acknowledgement}
The author would like to thank  Isaac Newton Institute for Mathematical Sciences, Cambridge, for support and hospitality during the programme \emph{Geometry, compatibility and structure preservation in computational differential equations (2019)} under Grant number EP/R014604/1,  where work on this paper was partly carried out.

The author would also like to thank the
European Union Horizon 2020 research and innovation programme under the Marie Sk\l{}odowska-Curie grant agreement  No. 691070 CHiPS.

\bibliographystyle{elsarticle-num-names} 
\bibliography{LIEEI}

\begin{thebibliography}{24}
\expandafter\ifx\csname natexlab\endcsname\relax\def\natexlab#1{#1}\fi
\providecommand{\url}[1]{\texttt{#1}}
\providecommand{\href}[2]{#2}
\providecommand{\path}[1]{#1}
\providecommand{\DOIprefix}{doi:}
\providecommand{\ArXivprefix}{arXiv:}
\providecommand{\URLprefix}{URL: }
\providecommand{\Pubmedprefix}{pmid:}
\providecommand{\doi}[1]{\href{http://dx.doi.org/#1}{\path{#1}}}
\providecommand{\Pubmed}[1]{\href{pmid:#1}{\path{#1}}}
\providecommand{\bibinfo}[2]{#2}
\ifx\xfnm\relax \def\xfnm[#1]{\unskip,\space#1}\fi
\bibitem[{Cary and Brizard(2009)}]{Cary}
\bibinfo{author}{J.~R. Cary}, \bibinfo{author}{A.~J. Brizard},
\newblock \bibinfo{title}{Hamiltonian theory of guiding-center motion},
\newblock \bibinfo{journal}{Rev. Modern Phys.} \bibinfo{volume}{81}
  (\bibinfo{year}{2009}) \bibinfo{pages}{693--738}. \URLprefix
  \url{https://doi.org/10.1103/RevModPhys.81.693}.
  \DOIprefix\doi{10.1103/RevModPhys.81.693}.
\bibitem[{Cotter and Reich(2006)}]{Cotter}
\bibinfo{author}{C.~J. Cotter}, \bibinfo{author}{S.~Reich},
\newblock \bibinfo{title}{Semigeostrophic particle motion and exponentially
  accurate normal forms},
\newblock \bibinfo{journal}{Multiscale Model. Simul.} \bibinfo{volume}{5}
  (\bibinfo{year}{2006}) \bibinfo{pages}{476--496}. \URLprefix
  \url{https://doi.org/10.1137/05064326X}. \DOIprefix\doi{10.1137/05064326X}.
\bibitem[{Hochbruck and Ostermann(2010)}]{Hochbruck2010}
\bibinfo{author}{M.~Hochbruck}, \bibinfo{author}{A.~Ostermann},
\newblock \bibinfo{title}{Exponential integrators},
\newblock \bibinfo{journal}{Acta Numer.} \bibinfo{volume}{19}
  (\bibinfo{year}{2010}) \bibinfo{pages}{209--286}. \URLprefix
  \url{https://doi.org/10.1017/S0962492910000048}.
  \DOIprefix\doi{10.1017/S0962492910000048}.
\bibitem[{Hochbruck et~al.(0809)Hochbruck, Ostermann, and
  Schweitzer}]{Hochbruck2008}
\bibinfo{author}{M.~Hochbruck}, \bibinfo{author}{A.~Ostermann},
  \bibinfo{author}{J.~Schweitzer},
\newblock \bibinfo{title}{Exponential {R}osenbrock-type methods},
\newblock \bibinfo{journal}{SIAM J. Numer. Anal.} \bibinfo{volume}{47}
  (\bibinfo{year}{2008/09}) \bibinfo{pages}{786--803}. \URLprefix
  \url{https://doi.org/10.1137/080717717}. \DOIprefix\doi{10.1137/080717717}.
\bibitem[{Hairer et~al.(2006)Hairer, Lubich, and Wanner}]{Hairer2006}
\bibinfo{author}{E.~Hairer}, \bibinfo{author}{C.~Lubich},
  \bibinfo{author}{G.~Wanner}, \bibinfo{title}{Geometric numerical
  integration}, volume~\bibinfo{volume}{31} of
  \textit{\bibinfo{series}{Springer Series in Computational Mathematics}},
  \bibinfo{edition}{second} ed., \bibinfo{publisher}{Springer-Verlag, Berlin},
  \bibinfo{year}{2006}. \bibinfo{note}{Structure-preserving algorithms for
  ordinary differential equations}.
\bibitem[{Celledoni et~al.(2008)Celledoni, Cohen, and Owren}]{Celledoni2008}
\bibinfo{author}{E.~Celledoni}, \bibinfo{author}{D.~Cohen},
  \bibinfo{author}{B.~Owren},
\newblock \bibinfo{title}{Symmetric exponential integrators with an application
  to the cubic {S}chr\"{o}dinger equation},
\newblock \bibinfo{journal}{Found. Comput. Math.} \bibinfo{volume}{8}
  (\bibinfo{year}{2008}) \bibinfo{pages}{303--317}. \URLprefix
  \url{https://doi.org/10.1007/s10208-007-9016-7}.
  \DOIprefix\doi{10.1007/s10208-007-9016-7}.
\bibitem[{Wu and Wang(2018)}]{Wuxinyuan2018}
\bibinfo{author}{X.~Wu}, \bibinfo{author}{B.~Wang}, \bibinfo{title}{Recent
  developments in structure-preserving algorithms for oscillatory differential
  equations}, \bibinfo{publisher}{Science Press Beijing, Beijing; Springer,
  Singapore}, \bibinfo{year}{2018}. \URLprefix
  \url{https://doi.org/10.1007/978-981-10-9004-2}.
  \DOIprefix\doi{10.1007/978-981-10-9004-2}.
\bibitem[{Li and Wu(2016)}]{Wuxinyuan2016}
\bibinfo{author}{Y.-W. Li}, \bibinfo{author}{X.~Wu},
\newblock \bibinfo{title}{Exponential integrators preserving first integrals or
  {L}yapunov functions for conservative or dissipative systems},
\newblock \bibinfo{journal}{SIAM J. Sci. Comput.} \bibinfo{volume}{38}
  (\bibinfo{year}{2016}) \bibinfo{pages}{A1876--A1895}. \URLprefix
  \url{https://doi.org/10.1137/15M1023257}. \DOIprefix\doi{10.1137/15M1023257}.
\bibitem[{Miyatake(2014)}]{miyatake2014energy}
\bibinfo{author}{Y.~Miyatake},
\newblock \bibinfo{title}{An energy-preserving exponentially-fitted continuous
  stage {R}unge--{K}utta method for {H}amiltonian systems},
\newblock \bibinfo{journal}{BIT Numerical Mathematics} \bibinfo{volume}{54}
  (\bibinfo{year}{2014}) \bibinfo{pages}{777--799}.
\bibitem[{Cui et~al.(2021)Cui, Xu, Wang, and Jiang}]{cui2021mass}
\bibinfo{author}{J.~Cui}, \bibinfo{author}{Z.~Xu}, \bibinfo{author}{Y.~Wang},
  \bibinfo{author}{C.~Jiang},
\newblock \bibinfo{title}{Mass-and energy-preserving exponential
  {R}unge--{K}utta methods for the nonlinear {S}chr{\"o}dinger equation},
\newblock \bibinfo{journal}{Applied Mathematics Letters} \bibinfo{volume}{112}
  (\bibinfo{year}{2021}) \bibinfo{pages}{106770}.
\bibitem[{Shen and Leok(2019)}]{shen2019geometric}
\bibinfo{author}{X.~Shen}, \bibinfo{author}{M.~Leok},
\newblock \bibinfo{title}{Geometric exponential integrators},
\newblock \bibinfo{journal}{Journal of Computational Physics}
  \bibinfo{volume}{382} (\bibinfo{year}{2019}) \bibinfo{pages}{27--42}.
\bibitem[{Jiang et~al.(2020)Jiang, Wang, and Cai}]{jiang2020linearly}
\bibinfo{author}{C.~Jiang}, \bibinfo{author}{Y.~Wang},
  \bibinfo{author}{W.~Cai},
\newblock \bibinfo{title}{A linearly implicit energy-preserving exponential
  integrator for the nonlinear {K}lein-{G}ordon equation},
\newblock \bibinfo{journal}{Journal of Computational Physics}
  \bibinfo{volume}{419} (\bibinfo{year}{2020}) \bibinfo{pages}{109690}.
\bibitem[{Shen et~al.(2019)Shen, Xu, and Yang}]{shen2019new}
\bibinfo{author}{J.~Shen}, \bibinfo{author}{J.~Xu}, \bibinfo{author}{J.~Yang},
\newblock \bibinfo{title}{A new class of efficient and robust energy stable
  schemes for gradient flows},
\newblock \bibinfo{journal}{SIAM Review} \bibinfo{volume}{61}
  (\bibinfo{year}{2019}) \bibinfo{pages}{474--506}.
\bibitem[{Furihata and Matsuo(2011)}]{furihata2011discrete}
\bibinfo{author}{D.~Furihata}, \bibinfo{author}{T.~Matsuo},
  \bibinfo{title}{Discrete variational derivative method}, Chapman \& Hall/CRC
  Numerical Analysis and Scientific Computing, \bibinfo{publisher}{CRC Press,
  Boca Raton, FL}, \bibinfo{year}{2011}. \bibinfo{note}{A structure-preserving
  numerical method for partial differential equations}.
\bibitem[{Dahlby and Owren(2011)}]{dahlby2011general}
\bibinfo{author}{M.~Dahlby}, \bibinfo{author}{B.~Owren},
\newblock \bibinfo{title}{A general framework for deriving integral preserving
  numerical methods for {PDE}s},
\newblock \bibinfo{journal}{SIAM J. Sci. Comput.} \bibinfo{volume}{33}
  (\bibinfo{year}{2011}) \bibinfo{pages}{2318--2340}. \URLprefix
  \url{https://doi.org/10.1137/100810174}. \DOIprefix\doi{10.1137/100810174}.
\bibitem[{Eidnes and Li(2020)}]{eidnes2020linearly}
\bibinfo{author}{S.~Eidnes}, \bibinfo{author}{L.~Li},
\newblock \bibinfo{title}{Linearly implicit local and global energy-preserving
  methods for {PDE}s with a cubic {H}amiltonian},
\newblock \bibinfo{journal}{SIAM Journal on Scientific Computing}
  \bibinfo{volume}{42} (\bibinfo{year}{2020}) \bibinfo{pages}{A2865--A2888}.
\bibitem[{Zhao et~al.(2017)Zhao, Wang, and Yang}]{zhao2017numerical}
\bibinfo{author}{J.~Zhao}, \bibinfo{author}{Q.~Wang},
  \bibinfo{author}{X.~Yang},
\newblock \bibinfo{title}{Numerical approximations for a phase field dendritic
  crystal growth model based on the invariant energy quadratization approach},
\newblock \bibinfo{journal}{International Journal for Numerical Methods in
  Engineering} \bibinfo{volume}{110} (\bibinfo{year}{2017})
  \bibinfo{pages}{279--300}.
\bibitem[{Hairer et~al.(2006)Hairer, Lubich, and Wanner}]{hairer2006geometric}
\bibinfo{author}{E.~Hairer}, \bibinfo{author}{C.~Lubich},
  \bibinfo{author}{G.~Wanner}, \bibinfo{title}{Geometric numerical integration:
  structure-preserving algorithms for ordinary differential equations},
  volume~\bibinfo{volume}{31}, \bibinfo{publisher}{Springer Science \& Business
  Media}, \bibinfo{year}{2006}.
\bibitem[{Eidnes et~al.(2019)Eidnes, Li, and Sato}]{eidnes2019linearly}
\bibinfo{author}{S.~Eidnes}, \bibinfo{author}{L.~Li},
  \bibinfo{author}{S.~Sato},
\newblock \bibinfo{title}{Linearly implicit structure-preserving schemes for
  {H}amiltonian systems},
\newblock \bibinfo{journal}{Journal of Computational and Applied Mathematics}
  (\bibinfo{year}{2019}) \bibinfo{pages}{112489}.
\bibitem[{Berland et~al.(2007)Berland, Skaflestad, and
  Wright}]{berland2007expint}
\bibinfo{author}{H.~Berland}, \bibinfo{author}{B.~Skaflestad},
  \bibinfo{author}{W.~M. Wright},
\newblock \bibinfo{title}{Expint---a {M}atlab package for exponential
  integrators},
\newblock \bibinfo{journal}{ACM Transactions on Mathematical Software (TOMS)}
  \bibinfo{volume}{33} (\bibinfo{year}{2007}) \bibinfo{pages}{4--es}.
\bibitem[{Hairer(2010)}]{hairer2010energy}
\bibinfo{author}{E.~Hairer},
\newblock \bibinfo{title}{Energy-preserving variant of collocation methods},
\newblock \bibinfo{journal}{Journal of Numerical Analysis, Industrial and
  Applied Mathematics} \bibinfo{volume}{5} (\bibinfo{year}{2010})
  \bibinfo{pages}{73--84}.
\bibitem[{McLachlan et~al.(1998)McLachlan, Quispel, and
  Robidoux}]{mclachlan1998unified}
\bibinfo{author}{R.~I. McLachlan}, \bibinfo{author}{G.~Quispel},
  \bibinfo{author}{N.~Robidoux},
\newblock \bibinfo{title}{Unified approach to {H}amiltonian systems, {P}oisson
  systems, gradient systems, and systems with {L}yapunov functions or first
  integrals},
\newblock \bibinfo{journal}{Physical Review Letters} \bibinfo{volume}{81}
  (\bibinfo{year}{1998}) \bibinfo{pages}{2399}.
\bibitem[{Mac{\'\i}as-D{\'\i}az and
  Medina-Ram{\'\i}rez(2009)}]{macias2009implicit}
\bibinfo{author}{J.~Mac{\'\i}as-D{\'\i}az},
  \bibinfo{author}{I.~Medina-Ram{\'\i}rez},
\newblock \bibinfo{title}{An implicit four-step computational method in the
  study on the effects of damping in a modified
  $\alpha$-{F}ermi--{P}asta--{U}lam medium},
\newblock \bibinfo{journal}{Communications in Nonlinear Science and Numerical
  Simulation} \bibinfo{volume}{14} (\bibinfo{year}{2009})
  \bibinfo{pages}{3200--3212}.
\bibitem[{Iavernaro and Trigiante(2009)}]{iavernaro2009high}
\bibinfo{author}{F.~Iavernaro}, \bibinfo{author}{D.~Trigiante},
\newblock \bibinfo{title}{High-order symmetric schemes for the energy
  conservation of polynomial {H}amiltonian problems},
\newblock \bibinfo{journal}{J. Numer. Anal. Ind. Appl. Math}
  \bibinfo{volume}{4} (\bibinfo{year}{2009}) \bibinfo{pages}{87--101}.

\end{thebibliography}





\end{document}